\documentclass[10pt,a4paper]{article}

\usepackage[english]{babel}
\usepackage{amscd,amsmath}
\usepackage{amssymb}
\usepackage{color}
\usepackage{amsfonts}
\usepackage{array}
\usepackage{graphicx}
\usepackage{float}
\usepackage{algorithm}
\usepackage{algorithmic}
\usepackage{mydef}
\usepackage{mathdots}
\usepackage{bm}
\usepackage{bbm,dsfont}
\usepackage{mathrsfs}
\usepackage{enumerate}
\usepackage[table]{xcolor}
\usepackage[section]{placeins}

\newenvironment{proof}{\noindent{\bf Proof. \/}\begin{small}\noindent}{\hfill\EndProofMarker\end{small}}
\newcommand{\EndProofMarker}{$\Box$}

\date{}
\begin{document}

\title{Least-Squares Pad\'e approximation of parametric and stochastic Helmholtz maps}

\author{Francesca Bonizzoni$^\sharp$\footnote{F. Bonizzoni has been funded by the Austrian Science Fund (FWF) through the project  F~65.}, 
	Fabio Nobile$^\mathsection$, 
	Ilaria Perugia$^\sharp$\footnote{I. Perugia has been funded by the Vienna Science and Technology Fund (WWTF) through the project MA14-006, 
	and by the Austrian Science Fund (FWF) through the projects P 29197-N32 and F~65.},
	Davide Pradovera$^\mathsection$}

\maketitle 

\vspace*{-0.6cm}
\begin{center}
{\small 
$^\sharp$ Faculty of Mathematics, University of Vienna\\
				Oskar-Morgenstern-Platz 1,1090 Wien, Austria\\
{\tt francesca.bonizzoni@univie.ac.at}, 
{\tt ilaria.perugia@univie.ac.at}\\ \medskip
$^\mathsection$ CSQI -- MATH, 
	Ecole Polytechnique F\'ed\'erale de Lausanne\\
	Station 8, CH-1015 Lausanne, Switzerland\\
{\tt fabio.nobile@epfl.ch}, 
{\tt davide.pradovera@epfl.ch}\\
}
\end{center}
\vspace*{0.2cm}

\begin{abstract}

The present work deals with the rational model order reduction method based on the single-point Least-Square (LS) Pad\'e approximation technique 
introduced in~\cite{Bonizzoni2016}. Algorithmical aspects concerning the construction of the rational LS-Pad\'e approximant are described. 
In particular, the computation of the Pad\'e denominator is reduced to the calculation of the eigenvector corresponding to the minimal 
eigenvalue of a Gramian matrix. The LS-Pad\'e technique is employed to approximate the frequency response map associated with various parametric 
time-harmonic acoustic wave problems, namely, a transmission/reflection problem, a scattering problem, and a problem in high-frequency regime. 
In all cases we establish the meromorphy of the frequency response map. The Helmholtz equation with stochastic wavenumber is also considered. 
In particular, for Lipschitz functionals of the solution and their corresponding probability measures, we establish weak convergence of the 
measure derived from the LS-Pad\'e approximant to the true one. 2D numerical tests are performed, which confirm the effectiveness of the approximation method.
\end{abstract}

\vspace*{0.2cm}


\vspace*{0.2cm}


\vspace*{0.2cm}

\section{Introduction}
\label{sec:Intruduction}

Many applications require the fast and accurate numerical evaluation of Helm\-holtz frequency response functions, i.e., functions that map the wavenumber to the solution (or some quantity of interest related to the solution) of the corresponding time-harmonic wave-problem, for a large number of frequencies. In mid- and high-frequency regimes, very fine meshes or high polynomial degrees should be considered, in order to obtain accurate Finite Element (FE) solutions of the time-harmonic wave-problem. Moreover, low order FE schemes are affected by the pollution effect~\cite{Babuska1997}, namely, an increasing discrepancy between
the best approximation error and the FE error, as the wave number increases. 
In the ``many-queries" context, i.e., when many solutions of the underlying Partial Differential Equation (PDE) are needed, the ``brute force" approach entails the solution of a large number of high-dimensional linear systems, and it is then out of reach.

Model order reduction methods aim at significantly reducing the computational cost by approximating the quantity of interest starting from evaluations at only few wavenumbers. 
They rely on a two-step strategy: the OFFLINE stage consists in the computation of a finite dimensional basis - e.g., the basis of snapshots 
(see, e.g., \cite{Chen2010,Huynh2007,Lassila2012,Sen2007,Sen2006,Tonn2011,Veroy2003,HUYNH2014,Hain2018, Modesto2015}), 
or evaluations of the frequency response map and its derivatives at fixed centers (Pad\'e method, see, e.g., \cite{Guillaume1998,Huard2000,Hetmaniuk2013,Hetmaniuk2012,
Bonizzoni2016}); the output of this phase, whose computational cost may be very high, is stored, to be used during the ONLINE phase, 
in which the approximation of the frequency response map corresponding to a given new value of the parameter is constructed.
This stage does not involve the numerical solution of
any PDE, and is expected to provide the output in real time. 

In this work, we focus on the Pad\'e-based model order reduction technique introduced in~\cite{Bonizzoni2016}, defined for any given univariate Hilbert 
space-valued meromorphic map $\T:\C\rightarrow V$, and relying on a \emph{single-point Least-Square (LS) Pad\'e approximant}. 
In particular, the single-point LS-Pad\'e approximant of $\T$ centered in $z_0\in\C$, denoted by $\pade{\T}$, is given by the rational $V$-valued map 
$\pade{\T}(z)=\frac{\num(z)}{\den(z)}$, where $\num(z)=\sum_{\alpha=0}^M p_\alpha (z-z_0)^\alpha$, with coefficients $p_\alpha\in V$ (we write $\num\in\PspaceV{M}{V}$), and $\den\in\PspaceS{N}$, where $\PspaceS{N}$ is the set of all polynomials with complex coefficients $\{q_\alpha\}_{\alpha=0}^N$ such that $\sum_{\alpha=0}^N\abs{q_{\alpha}}^2=1$.

In~\cite{Bonizzoni2016} we have analyzed the convergence of $\pade{\T}$ to $\T$ as $M\rightarrow \infty$ for a fixed denominator degree $N$. 
In particular, the LS-Pad\'e approximant $\pade{\T}$ identifies the $N$ poles of $\T$ closest to the center $z_0$, as limit of the roots of the denominator $\den(z)$ for $M$ going to $+\infty$.


In this paper, we describe in detail the \emph{algorithmical aspects} of the construction of the single-point LS-Pad\'e approximant. In particular, the identification of the 
LS-Pad\'e denominator is proved to be equivalent to the identification of the normalized eigenvector corresponding to the smallest non-negative eigenvalue of the Gramian matrix of the set $\left\{\T(z_0),\coeff{\T}{1,z_0},\ldots,\coeff{\T}{N,z_0}\right\}$, where $\coeff{\T}{\alpha,z_0}$ denotes the Taylor coefficient of $\T$ of order $\alpha$ at $z_0$.

Moreover, we explore the effectiveness of the single-point LS-Pad\'e technique when applied to \emph{parametric} frequency response problems which go beyond the 
setting considered in~\cite{Bonizzoni2016}, namely, a \emph{transmission/reflection} problem, and a \emph{scattering} problem. In both cases, 
we first prove that the frequency response map associated with the considered problem is meromorphic. 2D numerical results are provided, which demonstrate
the convergence of the LS-Pad\'e approximation. 
Moreover, 2D numerical tests in \emph{high-frequency} regime are performed for the parametric problem presented in~\cite{Bonizzoni2016}.

The \emph{stochastic} Helmholtz boundary value problem is also considered. We refer to~\cite{Ohayon2017} for Uncertainty Quantification for frequency responses in 
vibroacoustics, to~\cite{Ezvan2017, Jacquelin2017, Jacquelin} for model order reduction for random frequency responses in structural dynamics, 
and to~\cite{Schwab2011, Hiptmair2018} for the stochastic Helmholtz equation with uncertainty arising either in the forcing term or in the boundary data or in the shape of the scatterer.

Within the present framework, we propose a novel approach to the \emph{stochastic} Helm\-holtz boundary value problem based on the LS-Pad\'e technique, 
where the wavenumber $k^2$ is modeled as a random variable taking values into $K=[k^2_{min},k^2_{max}]$. 
We approximate the random variable $X:=\mathcal L(\mcS(k^2))$ with $X_P:=\mathcal L(\pade{\mcS}(k^2))$.
Here, $\mathcal L:V\rightarrow\R$ is a Lipschitz functional representing a quantity of interest, $\mcS$ is the meromorphic frequency response map associated with the (stochastic) Helmholtz equation endowed with either homogeneous  Dirichlet or homogeneous Neumann boundary conditions, and $\pade{\mcS}$ is the LS-Pad\'e approximation of $\mcS$. An upper bound on the approximation error for the characteristic function is derived.

All the considered boundary value problems fall into the following general setting.
Let $D$ be an open connected bounded Lipschitz domain in $\R^{d}$ ($d=1,2,3$), and  consider the following Helmholtz boundary value problem
\begin{equation}
\label{eq:helmholtz}
	\left\{\begin{array}{ll}
		-\Delta u-k^2 \varepsilon_r\ u=f&\textmd{in }D,\\
		u=g_D & \textmd{on }\Gamma_D,\\
		\Grad u\cdot\mathbf{n} = g_N & \textmd{on }\Gamma_N,\\
		\Grad u\cdot\mathbf{n} - i k u= g_R & \textmd{on }\Gamma_R,
	\end{array}\right.
\end{equation}
where the wavenumber $k^2$ is either a parameter or a random variable, which takes values into an interval of interest $K:=[k^2_{min},k^2_{max}]\subset\R^+$, $\varepsilon_r=\varepsilon_r(\x)\in L^\infty(D)$, $f\in L^2(D)$, $g_D\in H^{1/2}(\Gamma_D)$, $g_N\in H^{-1/2}(\Gamma_N)$, $g_R\in H^{-1/2}(\Gamma_R)$, and $\{\Gamma_D,\Gamma_N,\Gamma_R\}$ is a partition of $\partial D$, i.e., $\overline{\Gamma}_D\cup\overline{\Gamma}_N\cup\overline{\Gamma}_R=\partial D$ and $\Gamma_D\cap\Gamma_N=\Gamma_D\cap\Gamma_R=\Gamma_N\cap\Gamma_R=\emptyset$.
Throughout the paper, we denote with $V$ the Hilbert space $H^1_{\Gamma_D}(D)$. Moreover we assume the functions in $V$ to be complex-valued.

The outline of the paper is the following. In Section~\ref{sec:pade_approximant}, we recall the definition of the single-point LS-Pad\'e approximant and the main convergence 
result of~\cite{Bonizzoni2016}. In Section~\ref{sec:algorithmical_aspects}, we describe the algorithm to compute the LS-Pad\'e approximant. 
Section~\ref{sec:transmission_reflection} deals with a parametric transmission/reflection problem, whereas Section~\ref{sec:scattering_problem} deals with a parametric 
scattering problem. In Section~\ref{sec:high_frequency}, the LS-Pad\'e approximation is tested in high-frequency regime, and in Section~\ref{sec:stochastic_helmholtz} 
the LS-Pad\'e methodology is applied to the stochastic setting. Finally, conclusions are drawn in Section~\ref{sec:conclusions}.

%
%
%

\section{Least-Squares Pad\'e approximant of the parametric model problem}
\label{sec:pade_approximant}

This section deals with the Least-Squares (LS) Pad\'e approximation of the following parametric Helmholtz problem:
\begin{problem}[Parametric Model Problem]
	\label{prob:model_problem_parametric}
	The Helmholtz equation~\eqref{eq:helmholtz} has parametric wavenumber $k^2\in K:=[k^2_{min},k^2_{max}]\subset\R^+$, $\varepsilon_r=1$, and is endowed with either Dirichlet or Neumann homogeneous boundary conditions on $\partial D$, i.e., $\Gamma_R=\emptyset$ and either $\Gamma_D=\partial D$ and $g_D=0$, or $\Gamma_N=\partial D$ and $g_N=0$.
\end{problem}
The following result was proved in~\cite{Bonizzoni2016}.
\begin{theorem}
\label{thm:helmholtz_well_posed}
Let $\mcS$ be the frequency response map which associates to each $z\in \C$, the solution $u_z\in V$ of the weak formulation of Problem~\ref{prob:model_problem_parametric}:
\begin{align}
\label{eq:helmholtz_weak}
	& \int_D\Grad u_z(\x)\cdot\overline{\Grad v}(\x)\ d\x 
	-z\int_D u_z(\x)\overline{v}(\x)\ d\x
	= \int_D f(\x)\overline{v}(\x)\ d\x
	\quad\forall v\in V.
\end{align}
Then, $\mcS$ is well-defined, i.e., problem~\eqref{eq:helmholtz_weak} admits a unique solution for any $z\in\C\setminus\Lambda$, $\Lambda$ being the set of (real, non negative) eigenvalues of the Laplace operator with the considered boundary conditions.
Moreover, $\mcS$ is meromorphic in $\C$, with a pole of order one in each $\lambda\in\Lambda$. 
\end{theorem}
\begin{remark}
For the sake of simplicity, in Problem~\ref{prob:model_problem_parametric} we endow the Helmholtz equation with either homogeneous Dirichlet or homogeneous Neumann boundary conditions. Small modifications to the proofs of Theorem 3.1, Proposition 4.1, and Proposition 4.2 in~\cite{Bonizzoni2016} allow to handle both homogeneous mixed Dirichlet/Neumann and non-homogeneous Neumann boundary conditions, and to conclude an analogous result as Theorem~\ref{thm:helmholtz_well_posed}. In Section~\ref{sec:transmission_reflection}, we will show how to handle non-homogeneous Dirichlet boundary conditions.
\end{remark}

We recall now the definition and the convergence theorem of the LS-Pad\'e approximant of the frequency response map $\mcS$.

Let $K=[k^2_{min},k^2_{max}]\subset\R^+$ be the interval of interest, and $z_0\in \C\setminus\Lambda$ with $\Real{z_0}>0$. To fix the ideas we take $z_0=\frac{k^2_{min}+k^2_{max}}{2}+\delta i$, with $\delta\in\R\setminus\{0\}$ arbitrary.
The LS-Pad\'e approximant of $\mcS$, centered in $z_0$, is given by the ratio of two polynomials of degree $M$ and $N$ respectively:
\begin{equation}
\label{eq:pade}
	\pade{\mcS}(z):=\frac{\num(z)}{\den(z)}.
\end{equation}
The denominator $\den(z)$ is a function of $z$ only, and belongs to the space $\PspaceS{N}$ of all polynomials of degree at most $N$, $q=\sum_{i=0}^N q_i (z-z_0)^i\in\Pspace{N}$, such that $\sum_{i=0}^N\abs{q_i}^2=1$. The numerator $\num:\C\rightarrow V$ is a function of both the complex variable $z$ and the space variable $\x\in D$. More precisely, $\num(z)=\sum_{i=0}^M p_i(z-z_0)^i$, with coefficients $p_i\in V$. In the following, we denote with $\PspaceV{M}{V}$ the space of polynomials of degree at most $M$ in $z\in\C$ with coefficients in $V$.

The construction of the LS-Pad\'e approximant proposed in~\cite{Bonizzoni2016} relies on the minimization of the functional $j_{E,\rho}: \PspaceV{M}{V} \times \PspaceS{N} \rightarrow \R$, parametric in $E\in\N$ and $\rho\in\R^+$, defined as
\begin{equation}
\label{eq:functional_j}
\jfun{P}{Q}=\left(\sum_{\alpha=0}^E
	\normw{\coeff{Q(z)\mcS(z)-P(z)}{\alpha,z_0}}{V}{\sqrt{\Real{z_0}}}^2
	\rho^{2\alpha}\right)^{1/2},
\end{equation}
where the brackets $\coeff{\cdot}{\alpha,z_0}$ denote the $\alpha$-th Taylor coefficient of the Taylor series centered in $z_0$ (i.e., for a map $\T:\C\setminus\Lambda\rightarrow V$, $\coeff{\T(z)}{\alpha,z_0}=\frac{1}{\alpha!}\derzz{\alpha}{\T}(z_0)$), and $\normw{\cdot}{V}{\sqrt{\Real{z_0}}}$ denotes the weighted $H^1(D)$-norm (equivalent to the standard one) defined as
\begin{equation}
	\label{eq:weighted_norm}
	\normw{v}{V}{\sqrt{\Real{z_0}}}:=\sqrt{\norm{\Grad v}{ L^2(D)}^2 + \Real{z_0}\norm{v}{L^2(D)}^2}.	
\end{equation}

We recall the formal definition of the LS-Pad\'e approximant of the solution map $\mcS$, and we refer to Section~\ref{sec:algorithmical_aspects} for the proof of the existence of a (not in general unique) LS-Pad\'e approximant.
\begin{definition}
\label{def:pade_approximant}
Let $M,N\in \N$, $E\geq M+N$, and $\rho\in\R^+$. 
A LS-Pad\'e approximant $\pade{\mcS}$, centered in $z_0$, of the solution map $\mcS$ is a quotient $\frac{P}{Q}$ with $P\in\PspaceV{M}{V}$, $Q\in\PspaceS{N}$, such that
\begin{equation}
\label{eq:pade_approximant}
\jfun{P}{Q}\leq \jfun{R}{S}\quad\forall R\in\PspaceV{M}{V},\ \forall S\in\PspaceS{N}.
\end{equation}
\end{definition}

The following convergence result has been proved in~\cite{Bonizzoni2016}.

\begin{theorem}
\label{th:pade_conv}
Let $N\in\N$ be fixed, and let $R\in\R^+$ be such that the disk $\overline{\ball{z_0}{R}}$ contains exactly $N$ poles of $\mcS$. Then, for any $z\in\ball{z_0}{R}\setminus\Lambda$ and for any $\abs{z}<\rho<R$, it holds
\begin{equation*}
	\lim_{M\rightarrow \infty}
	\normw{\mcS(z)-\pade{\mcS}(z)}{V}{\sqrt{\Real{z_0}}}=0,
\end{equation*}
for all $E\geq M+N$.
Moreover, given $\alpha>0$ small enough, introduce the open subset 
$$
K_\alpha:=\bigcup_{\lambda\in\Lambda\cap K} (\lambda-\alpha,\lambda+\alpha)\subset K.
$$
Then for any $0<\rho<R$ such that $\ball{z_0}{\rho}\supset K$, there exists $M^\star\in \N$ such that, for any $M\geq M^\star$ and for any $z\in K\setminus K_\alpha$, it holds
\begin{equation}
\label{eq:pade_approx_S}
	\normw{\mcS(z)-\pade{\mcS(z)}}{V}{\sqrt{\Real{z_0}}}
	\leq C \frac{1}{\alpha^3} \left(\frac{\rho}{R}\right)^{M+1},
\end{equation}
where the constant $C>0$ depends on $\rho$, $R$, $N$, $z_0$, $\lambda_{min}=\min\{\lambda\in\Lambda\}$, $\norm{f}{L^2(D)}$, and  $g(z)=\prod_{\lambda\in\Lambda\cap \ball{z_0}{R}} (z-\lambda)$.
\end{theorem}

\begin{remark}
In~\cite{Bonizzoni2016} the bound
\begin{equation*}
	\normw{\mcS(z)-\pade{\mcS(z)}}{V}{\sqrt{\Real{z_0}}}
	\leq C \frac{1}{\alpha} \left(\frac{\rho}{R}\right)^{M+1}
\end{equation*}
was proved, with a constant $C$ that depends on $\frac{1}{g_{K\setminus K_\alpha}^2}$, with 
$$
g_{K\setminus K_\alpha}:=\min_{z\in K\setminus K_\alpha}\abs{g(z)},
$$ 
and $g(z)=\prod_{\lambda\in\Lambda\cap \ball{z_0}{R}} (z-\lambda)$. 
Since the frequency response map $\mcS$ presents only simple poles (given by the Dirichlet/Neumann Laplace eigenvalues), and the interval of interest $K$ contains a finite number of poles of $\mcS$, it follows that there exists $C_g>0$ such that 
$$
\abs{g(z)}\geq C_g\min_{\lambda\in\Lambda\cap \ball{z_0}{R}}\abs{z-\lambda} 
	\quad\forall z\in K,
$$ 
hence $g_{K\setminus K_{\alpha}}\geq C_g\alpha$ and the bound~\eqref{eq:pade_approx_S} follows.
\end{remark}

We can draw the following consequences:
\begin{enumerate}[(a)]
\item The roots of the LS-Pad\'e denominator $\den$ approximate the $N$ poles of $\mcS$, closest to $z_0$.
\item The region of convergence of $\pade{\mcS}$ is an open circle whose radius is equal to the distance between $z_0$ and the $(N+1)$-th closest pole of $\mcS$.
\end{enumerate}

\section{Algorithmical aspects}
\label{sec:algorithmical_aspects}

In this section, we describe an algorithm for the computation of a LS-Pad\'{e} approximant (defined according to Definition~\ref{def:pade_approximant}) of the Helmholtz frequency response map $\mcS$ introduced in the previous section. We underline that the presented algorithm can be likewise applied to any $V$-valued meromorphic map $\T:\C\rightarrow V$.
As a first instructive step, we recall the proof of the existence of such an approximant, which was developed in~\cite[Proposition 4.1]{Bonizzoni2016}.

\begin{proposition}
\label{prop:existence_pade}
For any $M,N\in \N$, $E\geq M+N$, and $\rho\in\R^+$, there exists a LS-Pad\'e approximant centered in $z_0$.
\end{proposition}

\begin{proof}
We want to show that the minimization problem~\eqref{eq:pade_approximant} admits at least one solution.
Since $P$ has degree $M$, then $\coeff{P(z)}{\alpha,z_0}=0$ for all $\alpha>M$. Hence, we can rewrite $j_{E,\rho}$ as
\begin{align*}
	\jfun{P}{Q}^2
	&=\sum_{\alpha=0}^M\normw{\coeff{Q(z)\mcS(z)-P(z)}{\alpha,z_0}}
		{V}{\sqrt{\Real{z_0}}}^2\rho^{2\alpha}\\
	&\quad+\sum_{\alpha=M+1}^E\normw{\coeff{Q(z)\mcS(z)-P(z)}{\alpha,z_0}}
		{V}{\sqrt{\Real{z_0}}}^2\rho^{2\alpha}\\
	&=\sum_{\alpha=0}^M\normw{\coeff{Q(z)\mcS(z)-P(z)}{\alpha,z_0}}
		{V}{\sqrt{\Real{z_0}}}^2\rho^{2\alpha}\\
	&\quad +\sum_{\alpha=M+1}^E\normw{\coeff{Q(z)\mcS(z)}{\alpha,z_0}}
		{V}{\sqrt{\Real{z_0}}}^2\rho^{2\alpha}.
\end{align*}
Now, let $Q$ be fixed. Taking $P=\bar P(Q)$, where $\bar P(Q)$ satisfies 
$$
\coeff{\bar P(z)}{\alpha,z_0}=\coeff{Q(z)\mcS(z)}{\alpha,z_0}
\quad\forall\ 0\leq\alpha\leq M,
$$ 
problem~\eqref{eq:pade_approximant} can be formulated as a minimization problem in $Q$ only: find $Q\in\PspaceS{N}$ such that
\begin{equation}
\label{eq:min_Q}
\jfunbar{Q}\leq \jfunbar{S}\quad\forall S\in\PspaceS{N},
\end{equation}
where 
\begin{equation}
\label{eq:jbar}
\jfunbar{Q}:=\jfun{\bar P(Q)}{Q}
	= \left(\sum_{\alpha=M+1}^E
	\normw{\coeff{Q(z)\mcS(z)}{\alpha,z_0}}
	{V}{\sqrt{\Real{z_0}}}^2\rho^{2\alpha}\right)^{1/2}.
\end{equation}
Since the functional $\bar j_{E,\rho}$ is continuous and the set $\PspaceS{N}$ is compact (being homeomorphic to the unit sphere in $\C^{N+1}$), $\bar j_{E,\rho}$ has a global minimum on $\PspaceS{N}$, and the minimization problem~\eqref{eq:min_Q} admits at least one solution.
\end{proof}

In the following proposition we express an equivalent formulation of the constrained minimization problem~\eqref{eq:min_Q}.

\begin{proposition}
\label{pr:step2}
The constrained minimization problem~\eqref{eq:min_Q} is equivalent to the identification of the (normalized) eigenvector corresponding to the smallest non-negative eigenvalue of the Hermitian positive-semidefinite matrix $\G\in\C^{(N+1)\times(N+1)}$ with entries
\begin{equation}
\label{eq:gram_matrix}
\left(\G\right)_{i,j}
	=\sum_{\alpha=M+1}^E
	\dualw{ \coeff{\mcS}{\alpha-j,z_0} }{ \coeff{\mcS}{\alpha-i,z_0} }
		\rho^{2\alpha},
	\qquad i,j=0,\ldots,N,
\end{equation}
where $\dualw{\cdot}{\cdot}$ denotes the scalar product that induces the weighted $H^1(D)$-norm $\normw{\cdot}{V}{\sqrt{\Real{z_0}}}$, and the Taylor coefficient of order $\beta$, $\coeff{\mcS}{\beta,z_0}$, is the unique solution of the following Helmholtz equation:
\begin{align}
	\nonumber
	& \int_D\Grad \coeff{\mcS}{\beta,z_0}(\x)\cdot\overline{\Grad v}(\x)\ d\x 
		-z\int_D \coeff{\mcS}{\beta,z_0}(\x)\overline{v}(\x)\ d\x\\
	\label{eq:coeff_S}
	&\quad
		= \int_D \coeff{\mcS}{\beta-1,z_0}\overline{v}(\x)\ d\x
		\quad \forall v\in V,
\end{align}
whereas we set $\coeff{\mcS}{\beta,z_0}=0$, whenever $\beta<0$.
\end{proposition}

\begin{proof}
Set $q_\alpha:=\coeff{Q}{\alpha,z_0}$ for $\alpha=0,\ldots,N$. 
Since 
$$
\coeff{Q\mcS}{\alpha,z_0}=\sum_{n=0}^\alpha q_n\coeff{\mcS}{\alpha-n,z_0}
=\sum_{n=0}^N q_n\coeff{\mcS}{\alpha-n,z_0}
$$ 
according to our convention that $\coeff{\mcS}{\beta,z_0}=0$ for $\beta<0$, we have
\begin{align*}
	\nonumber
\jfunbar{Q}^2= & \sum_{\alpha=M+1}^E
	\dualw{ \coeff{Q\mcS}{\alpha,z_0} } { \coeff{Q\mcS}{\alpha,z_0} }
	\rho^{2\alpha}\\
	\nonumber
=& \sum_{\alpha=M+1}^E
	\dualw{ \sum_{j=0}^{N} q_j \coeff{\mcS}{\alpha-j,z_0} }
	{ \sum_{i=0}^{N} q_i \coeff{\mcS}{\alpha-i,z_0} }
	\rho^{2\alpha}\\
	\nonumber
=& \sum_{\alpha=M+1}^E \sum_{i,j=0}^{N}
	q_i^* q_j \dualw{ \coeff{\mcS}{\alpha-j,z_0} }{ \coeff{\mcS}{\alpha-i,z_0} }
	\rho^{2\alpha}\\
= & \sum_{i,j=0}^N q_i^* q_j \sum_{\alpha=M+1}^E 
	\dualw{ \coeff{\mcS}{\alpha-j,z_0} }{ \coeff{\mcS}{\alpha-i,z_0} }
	\rho^{2\alpha}\\
= & \qq^\star \G \qq,
\end{align*}
where $\G\in\C^{(N+1)\times(N+1)}$ is defined in~\eqref{eq:gram_matrix}, and $\qq=(q_0,\ldots,q_N)^T$. 
By definition, $\G$ is Hermitian. Moreover, definition~\eqref{eq:jbar} implies that $\G$ is positive-semidefinite, so that all its eigenvalues are real non-negative. 
Finally, observe that the constraint $\sum_{\alpha=0}^N\abs{\coeff{Q}{\alpha,z_0}}^2=1$ is equivalent to the condition $\norm{\qq}{2}=1$. Hence, we conclude that the constrained minimization problem~\eqref{eq:min_Q} is equivalent to the identification of the (normalized) eigenvector corresponding to the smallest eigenvalue $\G$.
Finally, we observe that equation~\eqref{eq:coeff_S} is obtained by repeated differentiation of equation~\eqref{eq:helmholtz_weak}; see~\cite{Bonizzoni2016} for a rigorous derivation.
\end{proof}

The Hermitian matrix $\G$ defined in~\eqref{eq:gram_matrix} is obtained as weighted sum of sub-matrices of the Gram matrix $G\in\C^{(N+1)\times(N+1)}$ associated with the solution map $\mcS$, namely, the matrix with entries $G_{i,j}=\dualw{\coeff{\mcS}{i,z_0}}{\coeff{\mcS}{j,z_0}}$, for $i,j=0,\ldots,N$.
See Figure~\ref{fig:gram_matrix} for a graphical representation.
\begin{figure}[h]
\begin{center}
$G=\left[
\begin{tabular}{cccccc}
	$\dual{\mcS}{\mcS}$ & $\dual{\mcS}{\mcS_1}$ &
		$\dual{\mcS}{\mcS_2}$ & $\ldots$ & &\\
	$\dual{\mcS_1}{\mcS}$ & 	
		\cellcolor{blue!25}$\dual{\mcS_1}{\mcS_1}$ &
		\cellcolor{blue!25}$\dual{\mcS_1}{\mcS_2}$ &
		\cellcolor{blue!25}$\dual{\mcS_1}{\mcS_3}$ & 		
		$\ldots$\\
	$\dual{\mcS_2}{\mcS}$ & 		
		\cellcolor{blue!25}$\dual{\mcS_2}{\mcS_1}$ &
		\cellcolor{blue!25}$\dual{\mcS_2}{\mcS_2}$ &
		\cellcolor{blue!25}$\dual{\mcS_2}{\mcS_3}$ & 		
		$\dual{\mcS_2}{\mcS_4}$ &
		$\ldots$\\
	$\smash\vdots$ & 		
		\cellcolor{blue!25}$\dual{\mcS_3}{\mcS_1}$ &
		\cellcolor{blue!25}$\dual{\mcS_3}{\mcS_2}$ &
		\cellcolor{blue!25}$\dual{\mcS_3}{\mcS_3}$ & 		
		$\dual{\mcS_3}{\mcS_4}$ &
		$\ldots$\\
	& $\smash\vdots$ &
		$\dual{\mcS_4}{\mcS_2}$ &
		$\dual{\mcS_4}{\mcS_3}$ &
		$\dual{\mcS_4}{\mcS_4}$ &
		$\ldots$\\
	&& $\smash\vdots$ & $\smash\vdots$ & $\smash\vdots$ &
\end{tabular}
\right]$\\\vspace{0.2cm}
$\G=\ldots+\rho^6\left[
\begin{tabular}{ccc}
	$\dual{\mcS_3}{\mcS_3}$ & $\dual{\mcS_2}{\mcS_3}$ & $\dual{\mcS_1}{\mcS_3}$\\
	$\dual{\mcS_3}{\mcS_2}$ & $\dual{\mcS_2}{\mcS_2}$ & $\dual{\mcS_1}{\mcS_2}$\\
	$\dual{\mcS_3}{\mcS_1}$ & $\dual{\mcS_2}{\mcS_1}$ & $\dual{\mcS_1}{\mcS_1}$
\end{tabular}
\right]+\ldots$
\end{center}
\caption{Gram matrix (top) associated with the frequency response map $\mcS$. To lighten the notation, we omit both the argument ($z_0$) of the Taylor coefficients $\mcS_\alpha$, and the weight $\sqrt{\Real{z_0}}$ of the scalar product $\dual{\cdot}{\cdot}$. In blue the sub-matrix corresponding to $N=2$ and $\alpha=3$, which provides a contribution to $\G$ (bottom) with weight $\rho^6$. Observe that a transposition with respect to the secondary diagonal is carried out before computing the sum.}
\label{fig:gram_matrix}
\end{figure}

By following the steps performed in the proof of Proposition~\ref{prop:existence_pade}, and applying Proposition~\ref{pr:step2}, we devise Algorithm~\ref{al:pade} for the computation of the LS-Pad\'e approximant.
\begin{algorithm}[h]
\caption{Construction of the LS-Pad\'e approximant}
\label{al:pade}
\begin{algorithmic}[1]
\STATE Fix $z_0\in\C\setminus\Lambda$ with $\Real{z_0}>0$, $\rho\in\R^+$, 
	$M,\ N,\ E\in\N$, with $E\geq M+N$
\STATE Evaluate $\mcS$ in the center $z_0$, by solving problem~\eqref{eq:helmholtz_weak}
\FOR{$\beta=1,\ldots,E$}
	\STATE Compute the Taylor coefficient of $\mcS$ in $z_0$ of order $\beta$, $\coeff{\mcS}{\beta,z_0}$, by solving the problem~\eqref{eq:coeff_S}
\ENDFOR 
\STATE Define the matrix $\G\in\R^{(N+1)\times(N+1)}$ according to~\eqref{eq:gram_matrix}
\label{al:min_Q}
\STATE Compute the (normalized) eigenvector $\mathbf \xi=(\xi_0,\ldots,\xi_N)$ corresponding to the smallest non-negative eigenvalue of the matrix $\G$
\STATE Define the denominator as $\den(z)=\sum_{\alpha=0}^N \xi_\alpha (z-z_0)^\alpha$
\FOR{$\alpha=0,\ldots,N$}
	\STATE Compute the Taylor coefficient of $\mcS\den$ in $z_0$ of order $\alpha$ using the formula $\coeff{\mcS\den}{\alpha,z_0}=\sum_{n=0}^\alpha \xi_n\coeff{\mcS}{\alpha-n,z_0}$
\ENDFOR
\STATE Define the numerator as $\num(z)=\sum_{\alpha=0}^M\coeff{\mcS\den}{\alpha,z_0}(z-z_0)^\alpha$
\STATE Define the single-point LS-Pad\'e approximant as $\pade{\mcS}=\frac{\num}{\den}$
\end{algorithmic}
\end{algorithm}

\begin{remark}
\label{re:choice_rho}
The choice of $\rho$ impacts the algorithm only by determining the weights in the computation of $\G$. Specifically, small (respectively large) values of $\rho$ emphasize the contributions from the sub-matrices located in the top-left (respectively bottom-right) portion of $G$. A fast version of the algorithm, where $\G$ reduces just to the leading term (i.e., for $\rho\rightarrow+\infty$), is currently under investigation (see~\cite{Bonizzoni}).
\end{remark}

\section{Application to a transmission/reflection problem}
\label{sec:transmission_reflection}

We consider the transmission/reflection problem treated in~\cite{Kapita2015}, i.e., the transmission/reflection of a plane wave $e^{i\kappa \mathbf{x}\cdot\mathbf{d}}$ with wavenumber $\kappa$ and direction $\mathbf{d}=(\cos(\theta),\sin(\theta))$, across a fluid-fluid interface. In particular, the considered domain $D=(-1,1)^2$ is divided into two regions with different refractive indices $n_1,\ n_2$; we assume $n_1<n_2$. The Helmholtz problem is the following 
\begin{equation}
\label{eq:helmholtz_transmission}
	-\Delta u - \kappa^2\varepsilon_r^2 u =0,
	\quad\text{ with }
	\varepsilon_r(x_1,x_2)=\left\{\begin{array}{ll}
		n_1 & \text{ if }x_2<0,\\
		n_2 & \text{ if }x_2>0.
	\end{array}\right.
\end{equation}
For any angle $0\leq\theta<\pi/2$, the following function is a solution of equation~\eqref{eq:helmholtz_transmission}:
\begin{equation}
\label{eq:sol_transmission}
	u_{ex}(x_1,x_2)=\left\{\begin{array}{ll}
		T \exp\{i\mathbf{K}\cdot\mathbf{x}\} & \text{ if }x_2>0,\\
		\exp\{ i\kappa n_1 \mathbf{d}\cdot\x\} 
		+R \exp\{ i \kappa n_1 \mathbf{d}\cdot(x_1,-x_2)\} & \text{ if }x_2<0.\\
	\end{array}\right.
\end{equation}
where $\mathbf{K}=(\kappa n_1 d_1 , \kappa\sqrt{n_2^2-(n_1d_1)^2})$, $R=-\frac{K_2-\kappa n_1 d_2}{K_2+\kappa n_1 d_2}$ and $T=1+R$.
We couple the Helmholtz equation~\eqref{eq:helmholtz_transmission} with Dirichlet boundary conditions derived from the exact solution~\eqref{eq:sol_transmission}, i.e., $u|_{\partial D}=u_{ex}|_{\partial D}$. 

Depending on the value of $\theta$ (angle of the incident wave), the solution may exhibit two types of behavior:
\begin{itemize}
	\item if $\theta<\theta_{crit}:=\arccos\left(\frac{n_2}{n_1}\right)$, then $\Imag{K_2}\neq 0$, and $u_{ex}$ decays exponentially for $x_2>0$. Physically, this phenomenon is called \emph{total internal reflection};
	\item if $\theta>\theta_{crit}$, then $\mathbf{d}$ is close to the normal incidence, and the wave is \emph{refracted} at the interface.
\end{itemize}
The two behaviors are depicted in Figure~\ref{fig:transmission_theta_comparison}.

\begin{figure}[htb]
\centering
\includegraphics[width=0.45\textwidth]{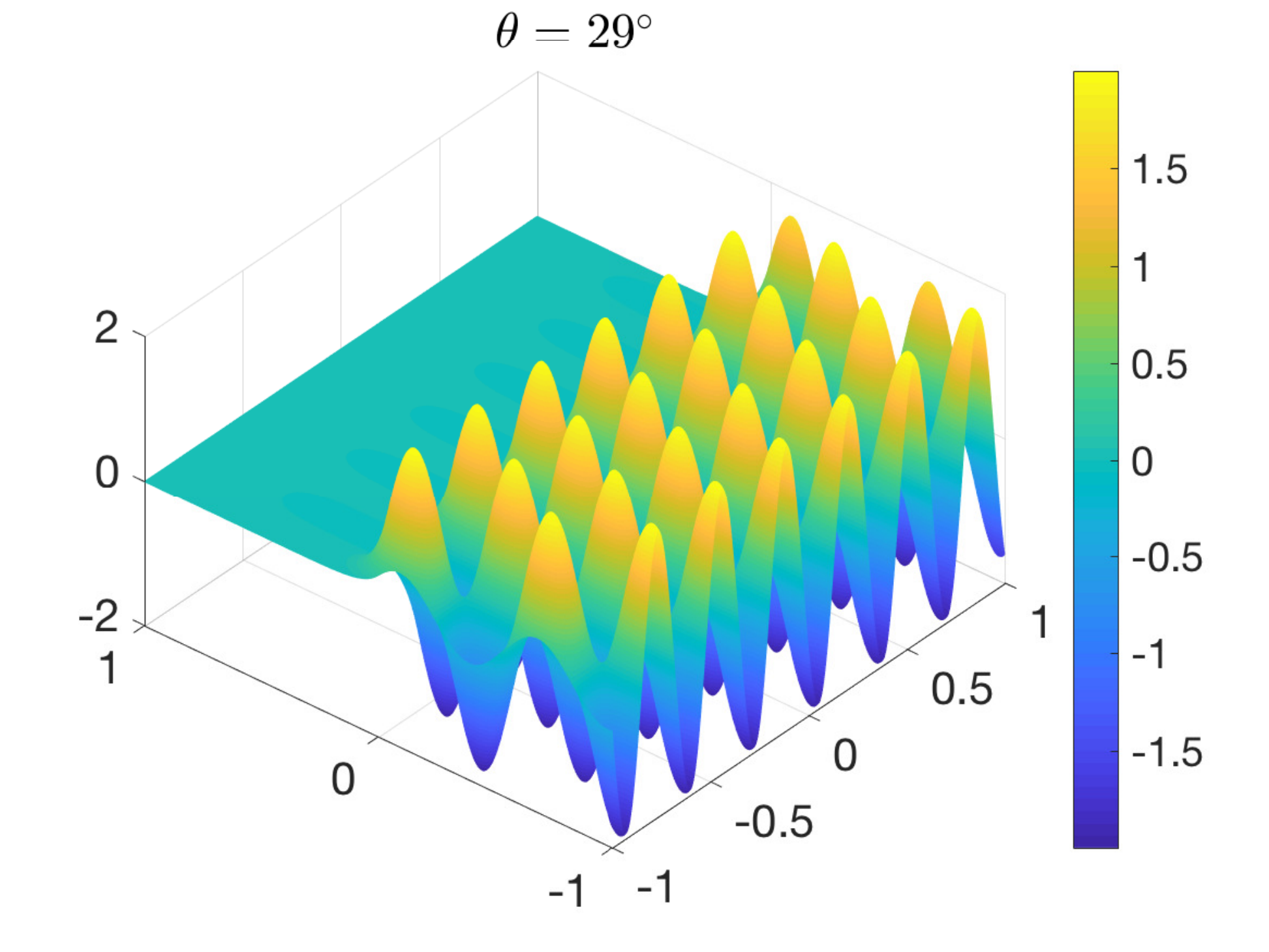}
\includegraphics[width=0.45\textwidth]{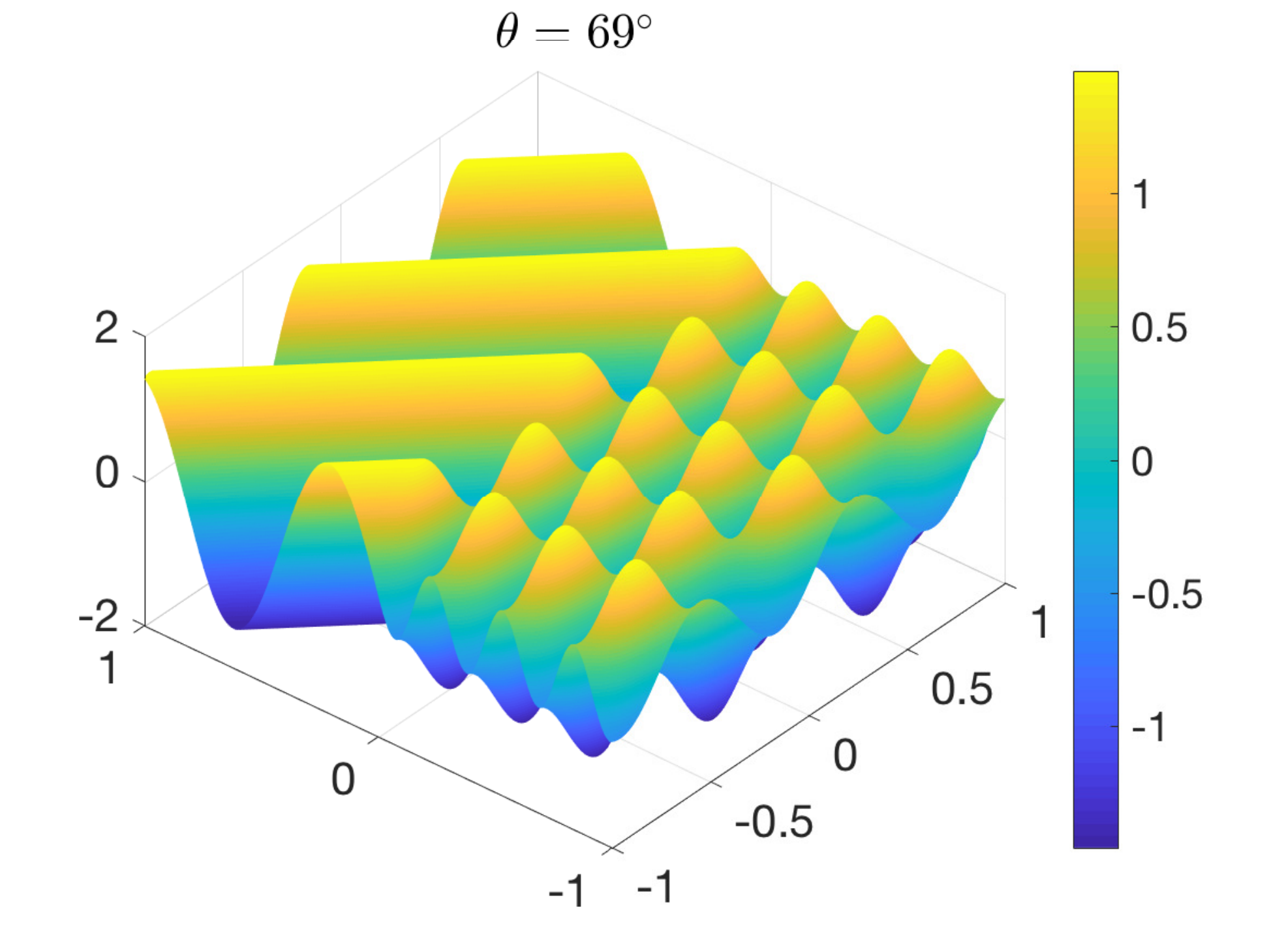}
\caption{Exact solution of the transmission/reflection problem with $n_1=2$, $n_2=1$, $\kappa=11$ and $\theta=29^\circ$ (left), $\theta=69^\circ$ (right).}
\label{fig:transmission_theta_comparison}
\end{figure}

\subsection{Frequency response map}

We are interested in the following boundary value problem:
\begin{problem}[Transmission/Reflection Problem] 
	\label{prob:transmission_reflection_bis}
	The wavenumber $\kappa^2$ ranges in the interval of interest $K=[\kappa^2_{min},\kappa^2_{max}]$, and the Helmholtz equation is endowed with Dirichlet boundary conditions on $\Gamma_D=\partial D$:
	\begin{equation}
	\label{eq:transmission_k_bis}
		\left\{\begin{array}{ll}
		-\Delta u - \kappa^2\varepsilon_r^2 u =0 &\text{ in }D,\\
		u=g_D & \text{ on }\partial D,
		\end{array}\right.
	\end{equation} 
	where $g_D:=u_{ex}|_{\partial D}$, and $u_{ex}$ is given by formula~\eqref{eq:sol_transmission} with $\kappa=11$ and either $\theta=29^\circ$ or $\theta=69^\circ$. 
\end{problem}
A weak formulation of problem~\eqref{eq:transmission_k_bis} with $z\in\C$ replacing $\kappa^2$ reads: find $\mathring u\in V=H^1_0(D)$ such that
\begin{align}
	\nonumber
	& \int_D \Grad \mathring{u}_z(\x)\cdot\overline{\Grad v}(\x) d\x
	- z \int_D \varepsilon_r^2(\x) \mathring{u}_z(\x)\overline v(\x) d\x \\
	\label{eq:transmission_weak_bis}
	&\quad = z \int_D \varepsilon_r^2(\x) w_g(\x)\overline v(\x) d\x
	\quad\forall v\in V,
\end{align}
where $w_g\in H^1(D)$ is the unique harmonic extension of $g_D$, i.e., $\Delta w_g=0$ in $D$ and $w_g|_{\partial D}=g_D$, and $\mathring{u}:=u-w_g$.

By generalizing~\cite[Theorem 2.1]{Bonizzoni2016}, it can be proved that problem~\eqref{eq:transmission_weak_bis} admits a unique solution for all $z\in\C\setminus\Lambda$, $\Lambda$ being the set of eigenvalues of the Laplacian (w.r.t. the weighted $L^2(D)$-norm $\norm{v}{L^2(D),\varepsilon_r}=\norm{\varepsilon_r v}{L^2(D)}$)  with homogeneous Dirichlet boundary conditions. Moreover, with
\begin{equation}
		\label{eq:transmission_alpha_bis}
		0<\alpha<\min_{j:\lambda_j\in\Lambda}\abs{\lambda_j-z},
\end{equation}
the unique solution satisfies the a priori bound
\begin{equation}
		\label{eq:transmission_bound_u_bis}
		\normw{\mathring{u}_z}{V}{\sqrt{\Real{z_0}}}\leq 
		\max\{1,n_1,n_2\}
		\frac{\sqrt{\abs{\lambda_{min}-z} + \abs{\Real{z}} +\Real{z_0} }}{\alpha}
		\abs{z}\norm{w_g}{L^2(D)},
\end{equation}
where $\lambda_{min}:=\min\{\lambda\in\Lambda\}$. 
By triangular inequality, an analogous upper bound on $\normw{u_z}{V}{\sqrt{\Real{z_0}}}$ follows.

Let us denote by $\mcS:\C\rightarrow V:=H^1(D)$ the frequency response map that associates to each complex wavenumber $z$, the function $\mcS(z)=\mathring{u}_z+w_g$, with $\mathring{u}_z$ the weak solution of~\eqref{eq:transmission_weak_bis}.
\begin{proposition}
	\label{prop:transmission_S_meromorphy_bis}
	The frequency response map $\mcS$ is meromorphic in $\C$, having a pole of order one in each $\lambda\in\Lambda$, where $\Lambda$ is the set of eigenvalues of the Laplacian (w.r.t. the weighted $L^2(D)$-norm $\norm{\cdot}{L^2(D),\varepsilon_r}$) with homogeneous Dirichlet boundary conditions.
\end{proposition}

\begin{proof}
	We denote with $\duale{\cdot}{\cdot}{\varepsilon_r}$ the inner product which induces the $L^2(D)$ weighted norm $\norm{\cdot}{L^2(D),\varepsilon_r}$, i.e.,  $\duale{v_1}{v_2}{\varepsilon_r}:=\int_D \varepsilon_r^2(\x) v_1(\x) v_2(\x) d\x$. 
	Let $\{\varphi_j\}$ be the set of eigenfunctions of the Laplacian (with homogeneous Dirichlet boundary conditions) orthonormal with respect to the inner product $\duale{\cdot}{\cdot}{\varepsilon_r}$, and let $\{\lambda_j\}$ be the corresponding eigenvalues, i.e., $-\Delta\varphi_j=\lambda_j\varepsilon^2_r\varphi_j$ in $D$ and $\varphi_j|_{\partial D}=0$ (see, e.g., \cite[Theorem 2.36]{McLean2000}).
Inserting into equation~\eqref{eq:transmission_weak_bis} the eigenfunction expansion $\mathring{u}(z,\x)=\sum_{j}\mathring{u}_j(z)\varphi_j(\x)$, where $\mathring{u}_j(z):=\duale{\mathring{u}(z)}{\varphi_j}{\varepsilon_r}$, and denoting $w_j:=\duale{w_g}{\varphi_j}{\varepsilon_r}$, we derive
\begin{equation}
	\label{eq:trasmission_u_coeff_bis}
	\mathring{u}_j(z)=\frac{z\ w_j}{\lambda_j-z}.
\end{equation}
The eigenfunction expansion of the frequency response map is then given by
\begin{align}
	\nonumber
	\mcS(z)
	& =\mathring{u}(z,\x)+w_g(\x)
	=\sum_j \mathring{u}_j(z)\varphi_j(\x) + w_g(\x)\\
	\label{eq:S_eig_bis}
	& \stackrel{\eqref{eq:trasmission_u_coeff_bis}}{=}
		\sum_j \frac{zw_j}{\lambda_j-z}\varphi_j(\x) + w_g(\x).
\end{align}
Since the series converges in the (weighted) $H^1(D)$-norm, then~\eqref{eq:S_eig_bis} directly implies that $\mcS$ is meromorphic in $\C$, and each $\lambda\in\Lambda$ is a pole of order one for $\mcS$.
\end{proof}

\subsection{LS-Pad\'e approximant of the frequency response map}

Since the frequency response map is meromorphic, it is appropriate to use the LS-Pad\'e technology to catch the singularities of $\mcS$, and provide sharp approximations of $\mcS(z)$, when $z$ is close to the center $z_0$. We apply Algorithm~\ref{al:pade}, and  compute the coefficients of the denominator as the entries of the eigenvector corresponding to the minimal eigenvalue of the Gram matrix~\eqref{eq:gram_matrix}. 
The Taylor coefficient of order $\beta\geq 1$, $\coeff{\mcS}{\beta,z_0}=\frac{1}{\beta!}\derzz{\beta}{\mcS}|_{z=z_0}\in H^1_0(D)$, satisfies
\begin{align}
	\nonumber
	& \int_D \Grad \coeff{\mcS}{\beta,z_0}(\x)\cdot\overline{\Grad v}(\x) d\x
			- z_0 \int_D \varepsilon_r^2(\x) \coeff{\mcS}{\beta,z_0}(\x)\overline v(\x) d\x\\
	\label{eq:transmission_beta}
	& = \int_D \varepsilon_r^2(\x) \coeff{\mcS}{\beta-1,z_0}(\x)\overline v(\x) d\x
	\quad\forall v\in H^1_0(D).
\end{align}
Problem~\eqref{eq:transmission_beta} admits a unique solution for all $z\in\C\setminus\Lambda$, since the PDE operator is the same as in~\eqref{eq:transmission_weak_bis} and the right-hand side is a bounded linear form.



Let $K=[3,12]$ be the interval of interest and $\theta=29^\circ$. In Figure~\ref{fig:transmission_pade_norm}, the $H^1(D)$-weighted norm of the $\mathbb P^2$ finite element approximation of $\mcS$, $\mcS_h$, is compared with the norm of its LS-Pad\'e approximant $\mcS_{h,P}$ centered in $z_0=7.5+0.5i$, for various degrees $(M,N)$.
We have empirically observed (see Figure~\ref{fig:transmission_error}) that the LS-Pad\'e approximation delivers a better accuracy than that predicted in~\eqref{eq:pade_approx_S}:
\begin{equation} 
	\label{eq:numerical_rate}
	\normw{\mcS_h(z)-\padeh{\mcS}(z)}{V}{\sqrt{\Real{z_0}}}
	\sim\left(\frac{\abs{z_0-z}}{\abs{z_0-\lambda_{N+1}}}\right)^{M+1},
\end{equation}
where $\{\lambda_j\}_j$ are the elements of $\Lambda$ ordered according to: $\abs{\lambda_1-z_0}<\abs{\lambda_2-z_0}<\ldots$. 
We refer to~\cite{Bonizzoni} for a formal derivation of~\eqref{eq:numerical_rate}, where $\pade{\mcS}$ is computed by a fast version of Algorithm~\ref{al:pade}.

\begin{figure}[h]
\centering
\includegraphics[width=0.45\textwidth]{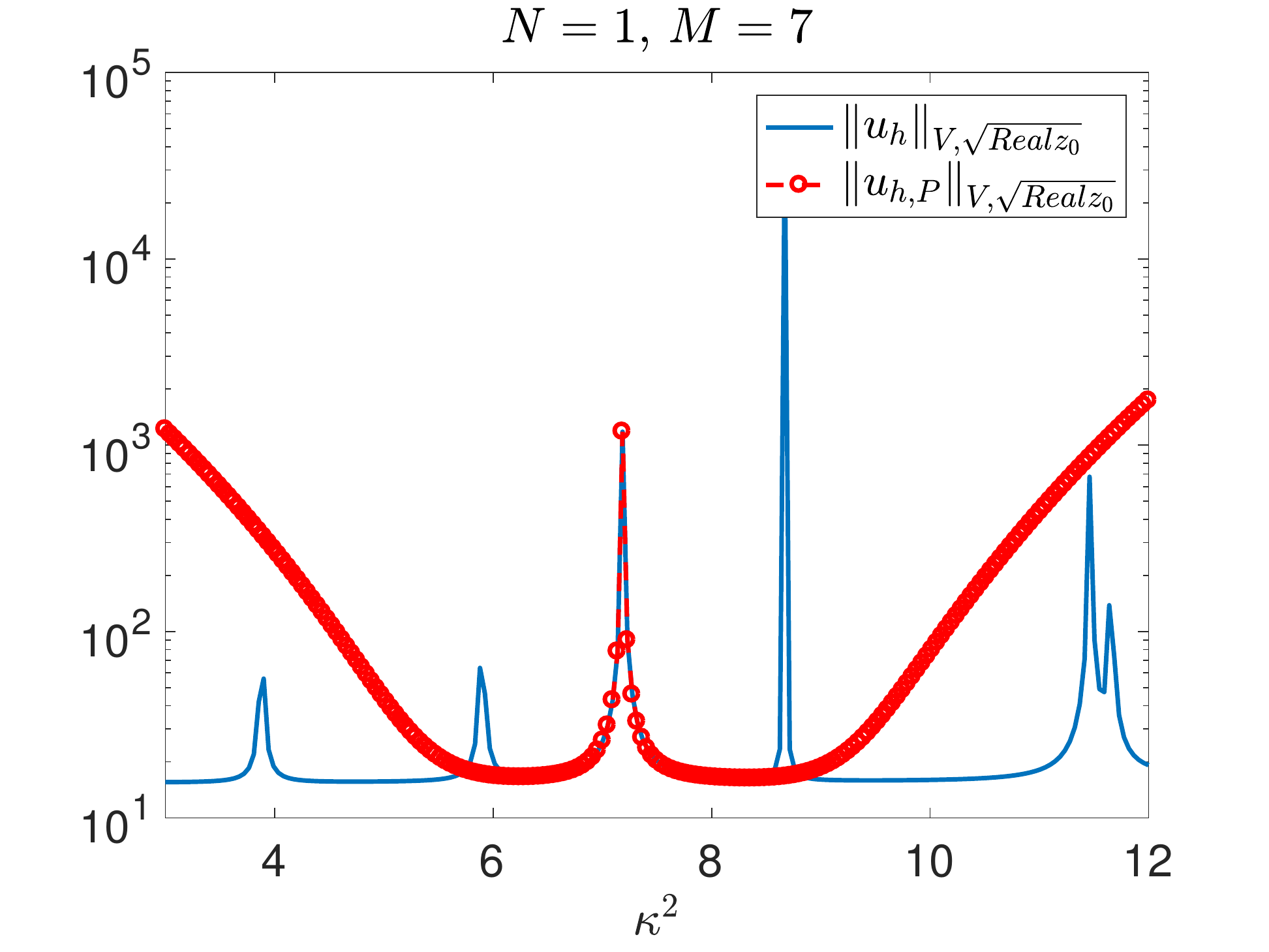}
\includegraphics[width=0.45\textwidth]{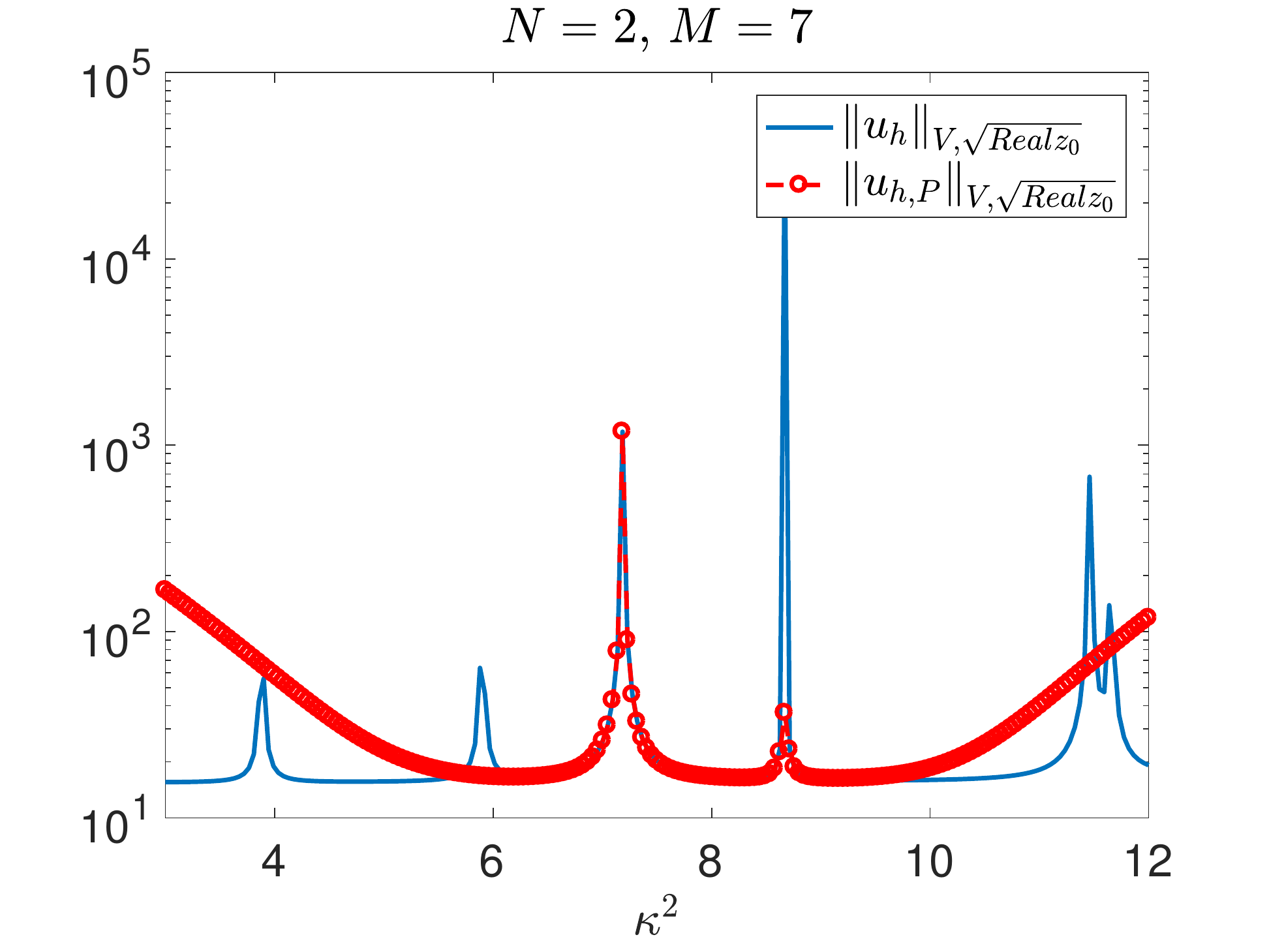}
\includegraphics[width=0.45\textwidth]{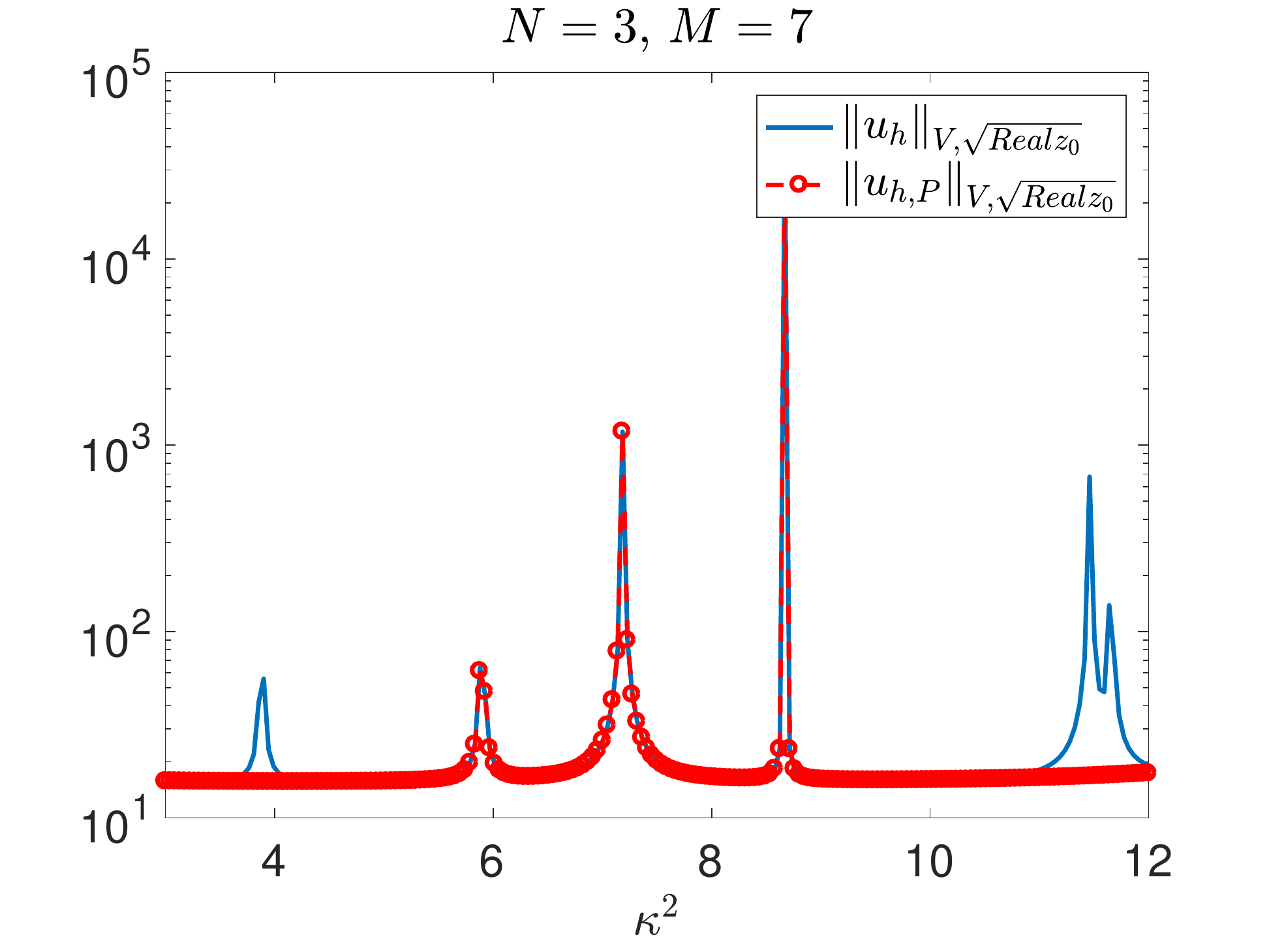}
\caption{Comparison between the $H^1(D)$-weighted norm of $\mcS_h$ (with $\theta=29^\circ$) and of its LS-Pad\'e approximant $\mcS_{h,P}$ centered in $z_0=7.5+0.5i$.}
\label{fig:transmission_pade_norm}
\end{figure}

\begin{figure}[h]
\centering
\includegraphics[width=0.6\textwidth]{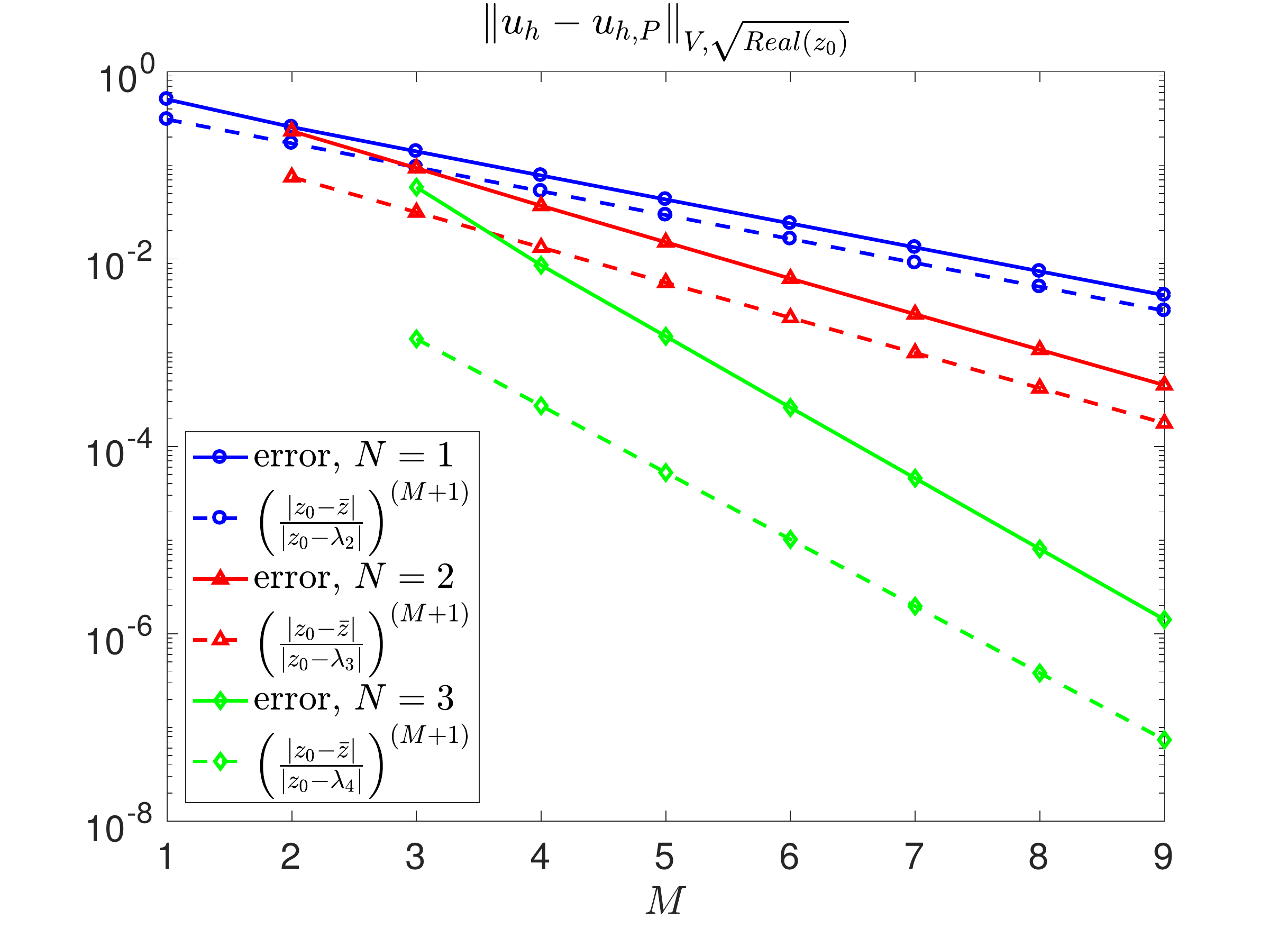}
\caption{LS-Pad\'e approximation error, compared with the heuristic slope $\left(\frac{\abs{z_0-\bar z}}{\abs{z_0-\lambda_{N+1}}}\right)^{M+1}$ for $\bar z=8$.}
\label{fig:transmission_error}
\end{figure}

\section{Application to a scattering problem}
\label{sec:scattering_problem}

In this section, we consider the scattering of an acoustic wave at a scatter occupying the domain $\ball{(0,0)}{0.5}\subset\R^2$. The incident wave $u^i$ is the time-harmonic plane wave traveling along the direction $\d=(\cos(\theta),\sin(\theta))$ with wavenumber $k$, i.e., $u^i=e^{ik\d\cdot\x}$. The total field $u$, given by the sum of the incident wave $u^i$ with the scattered wave $u^s$, satisfies the following boundary value problem in the infinite domain $\R^2\setminus\ball{(0,0)}{0.5}\subset\R^2$
\begin{equation}
	\label{eq:total_field}
	\left\{\begin{array}{ll}
		-\Delta u - k^2 u = 0 &\text{ in }\R^2\setminus \overline{\ball{(0,0)}{0.5}\subset\R^2},\\
		u=0 & \text{ on }\Gamma_D:=\partial\ball{(0,0)}{0.5}\subset\R^2,\\
		\lim_{\abs{\x}\rightarrow \infty} \abs{x}^{1/2}\left(
		\frac{\partial u^s(\x)}{\partial\abs{x}}-iku^s(\x)\right)=0
	\end{array}
	\right.
\end{equation}
The finite element approximation of problem~\eqref{eq:total_field} entails the truncation of the unbounded domain $\R^2\setminus \overline{\ball{(0,0)}{0.5}\subset\R^2}$ into the bounded domain 
$$
D:=\left([-2,2]\times[-2,2]\right)\setminus\overline{\ball{0}{0.5}},
$$ 
whose outer boundary will be denoted as $\Gamma_R$. Approximating the Sommerfeld radiation condition at infinity in problem~\eqref{eq:total_field} by a first order absorbing boundary condition, we write the following parametric problem:
\begin{problem}[Scattering Problem]
The wavenumber $k^2$ ranges in the interval of interest $K:=[k^2_{min},k^2_{max}]\subset\R^+$, $\mathbf n$ is the outgoing normal vector field to $\Gamma_R$, and $g_R:=\frac{\partial u^i}{\partial \mathbf n} - iku^i$ is the impedance trace of the incoming wave $u^i$. We consider the Helmholtz boundary value problem 
\begin{equation}
	\label{eq:total_field_D}
	\left\{\begin{array}{ll}
		-\Delta u - k^2 u = 0 &\text{ in }D,\\
		u=0 & \text{ on }\Gamma_D,\\
		\frac{\partial u}{\partial \mathbf n} - iku = g_R & \text{ on }\Gamma_R.
	\end{array}
	\right.
\end{equation}
\end{problem}

\subsection{Regularity of the frequency response map}

We extend problem~\eqref{eq:total_field_D} to complex wavenumbers. Given a complex wavenumber $z\in\C$, we introduce the incident plane wave $u^i=e^{iz\d\cdot\x}$ and its impedance trace $g_z:=\frac{\partial u^i}{\partial \mathbf n} - izu^i$, and we define the frequency response map $\mcS:z\mapsto\mcS(z):=u_z\in V:=H^1_{\Gamma_D}(D)$, where $u_z$ satisfies
\begin{align}
	\label{eq:total_field_weak}
	&\int_D\Grad u_z(\x)\cdot\overline{\Grad v}(\x) d\x
	-z^2 \int_D u_z(\x)\overline{v}(\x) d\x
	-iz \int_{\Gamma_R} u_z(\x)\overline{v}(\x) ds\\
	\nonumber
	&\quad=\int_{\Gamma_R} g_z(\x)\overline{v}(\x) ds
	\quad\forall\ v\in V.
\end{align} 
If $z\in\R$, problem~\eqref{eq:total_field_weak} admits a unique solution (see, e.g., \cite{Hiptmair2014}), which implies that the frequency response map is well-defined on $\R$. The following Theorem extends this result to the complex half plane $\left\{z\in\C:\ \Imag{z}\geq 0\right\}$.
Since the wavenumber in~\eqref{eq:total_field_weak} is square of the parameter $z$, we will endow the Hilbert space $V$ with the weighted $H^1(D)$-norm, with weight $w=\Real{z_0}$ (and not $w=\sqrt{\Real{z_0}}$, as was done before).

\begin{theorem}
	\label{thm:scattering_wp}
	Problem~\eqref{eq:total_field_weak} admits a unique solution in all compact subsets of 
	\begin{equation}
		\label{eq:C+}
		\C^+:=\left\{z\in\C:\ \Imag{z}\geq 0\right\}.
	\end{equation}
\end{theorem}

\begin{proof}
Given $z\in\C$, we introduce the bilinear and linear forms which define problem~\eqref{eq:total_field_weak}:
\begin{align}
	\label{eq:scattering_B}
	B_z(u,v)& :=\int_D\Grad u_z(\x)\cdot\overline{\Grad v}(\x) d\x
	-z^2 \int_D u_z(\x)\overline{v}(\x) d\x
	-iz \int_{\Gamma_R} u_z(\x)\overline{v}(\x) ds,\\
	\label{eq:scattering_L}
	L_z(v)& :=\int_{\Gamma_R} g_z(\x)\overline{v}(\x) ds.
\end{align}
We first show that either the coercivity or the G\aa rding inequality (see~\cite{McLean2000}) holds, provided that $\Imag{z}$ is non-negative. For the bilinear form in~\eqref{eq:scattering_B}, we have
\begin{align*}
	\Real{B_z(u,u)}
	& =  \norm{\Grad u}{L^2(D)}^2
		-(\Real{z}^2-\Imag{z}^2) \norm{u}{L^2(D)}^2 
		+\Imag{z} \norm{u}{L^2(\Gamma_R)}^2 \\
	& \geq \norm{\Grad u}{L^2(D)}^2
		-(\Real{z}^2-\Imag{z}^2) \norm{u}{L^2(D)}^2 \\
	& = \normw{u}{V}{\Real{z_0}}^2
		-(\Real{z}^2-\Imag{z}^2+\Real{z_0}^2) \norm{u}{L^2(D)}^2.
\end{align*}
If $C:=\Real{z}^2-\Imag{z}^2+\Real{z_0}^2\leq 0$, then $B(\cdot,\cdot)$ is coercive, whereas if $C>0$, then $B_z(\cdot,\cdot)$ satisfies the G\aa rding inequality. 

The bilinear form~\eqref{eq:scattering_B} is bounded, with constant $C=\max\left\{1,\frac{\abs{z}^2}{\Real{z_0}},
		\frac{\abs{z}C_{tr}^2}{\Real{z_0}}\right\}$. 
Indeed, using the trace inequality
\begin{equation*}
	\norm{u}{L^2(\Gamma_R)}\leq C_{tr}\norm{u}{H^1(D)},
\end{equation*}
we get
\begin{align*}
	\abs{B_z(u,v)}
	& \leq \norm{\Grad u}{L^2(D)}\norm{\Grad v}{L^2(D)}
		+\abs{z}^2\norm{u}{L^2(D)}\norm{v}{L^2(D)}
		+\abs{z}\norm{u}{L^2(\Gamma_R)}\norm{v}{L^2(\Gamma_R)}\\
	& \leq \norm{\Grad u}{L^2(D)}\norm{\Grad v}{L^2(D)}
		+\abs{z}^2\norm{u}{L^2(D)}\norm{v}{L^2(D)}
		+\abs{z} C_{tr}^2 \norm{u}{H^1(D)}\norm{v}{H^1(D)}\\
	& \leq \norm{\Grad u}{L^2(D)}\norm{\Grad v}{L^2(D)}
		+\frac{\abs{z}^2\Real{z_0}^2}{\Real{z_0}^2}\norm{u}{L^2(D)}\norm{v}{L^2(D)}\\
	&\quad+\abs{z} C_{tr}^2 \max\left\{1,\frac{1}{\Real{z_0}^2}\right\}
			\normw{u}{V}{\Real{z_0}}\norm{v}{\Real{z_0}}\\
	& \leq \max\left\{1,\frac{\abs{z}^2}{\Real{z_0}^2},
		\frac{\abs{z}C_{tr}^2}{\Real{z_0}^2}\right\}
		\normw{u}{V}{\Real{z_0}}\norm{v}{\Real{z_0}}.
\end{align*}
Moreover, the linear functional~\eqref{eq:scattering_L} is bounded, with constant
$$
C=C_{tr}^2\max\left\{1,\frac{1}{\Real{z_0}^2}\right\}\normw{g_z}{V}{\Real{z_0}}.
$$
Problem~\eqref{eq:total_field_weak} admits a unique solution (continuously dependent on the data) if and only if its homogeneous adjoint problem admits only trivial solutions: see~\cite[Theorem 4.11]{McLean2000}. We consider the case $\Imag{z}> 0$, and we refer to~\cite{Hiptmair2014} for $\Imag{z}=0$.
The bilinear form associated with the adjoint problem with $g_z=0$ reads:
\begin{equation*}
	B^*_z(\varphi,v):=\overline{B_z(v,\varphi)}=
	\int_D\Grad \varphi(\x)\cdot\overline{\Grad v}(\x) d\x
	-\overline{z}^2 \int_D \varphi(\x)\overline{v}(\x) d\x
	- \overline{iz} \int_{\Gamma_R} \varphi(\x)\overline{v}(\x) ds,
\end{equation*}
and the condition $B^*_z(u,u)=0$ is equivalent to
\begin{equation*}
	\left\{\begin{array}{l}
	\Real{B^*_z(u,u)}=\norm{\Grad u}{L^2(D)}^2 
		- (\Real{z}^2-\Imag{z}^2)\norm{u}{L^2(D)}^2 
		+ \Imag{z}\norm{u}{L^2(\Gamma_R)}^2=0\\
	\Imag{B^*_z(u,u)} = \Real{z}\left(2\Imag{z}\norm{u}{L^2(D)}^2
		+\norm{u}{L^2(\Gamma_R)}^2 \right)=0
	\end{array}\right.
\end{equation*}
If $\Real{z}\neq 0$ and $\Imag{z}> 0$, then $\Imag{B^*_z(u,u)}=0$ is equivalent to $\norm{u}{L^2(D)}=\norm{u}{L^2(\Gamma_R)}=0$, that is, $u=0$ in $D$,
whereas, if $\Real{z}=0$ and $\Imag{z}>0$, then $\Real{B^*_z(u,u)}=0$ implies $\norm{\Grad u}{L^2(D)}=\norm{u}{L^2(D)}=\norm{u}{L^2(\Gamma_R)}=0$, hence $u=0$.
\end{proof}

We recall here the following theorem, see~\cite[Theorem 1]{Steinberg1968}, which will be used in the proof of Proposition~\ref{prop:scattering_mer}.
\begin{theorem}
	\label{thm:steinberg}
	Let $B$ be an open and connected subset of the complex plane.
	If $\{T(z)\}_{z\in B}$ is an analytic family of compact operators defined on a given Banach space, then either $(I-T(z))$ is nowhere invertible in $B$ or $(I-T(z))^{-1}$ is meromorphic in $B$.
\end{theorem}

\begin{proposition}
	\label{prop:scattering_mer}
	The frequency response map $\mcS$ associated with problem~\eqref{eq:total_field_weak} is meromorphic in all open bounded and connected subsets of $\C$, and all its poles have negative imaginary part.
\end{proposition}

\begin{proof}
We proceed as in~\cite[Proposition 2]{Lenoir1992}.
We add and subtract the term $\int_D u_z\overline v dx$ to the left-hand side of~\eqref{eq:total_field_weak}, and we get
\begin{align*}
	&\int_D\Grad u_z(\x)\cdot\overline{\Grad v}(\x) d\x
	+\int_D u_z(\x)\overline{v}(\x) dx
	-(1+z^2) \int_D u_z(\x)\overline{v}(\x) d\x\\
	&\quad-iz \int_{\Gamma_R} u_z(\x)\overline{v}(\x) ds
	=\int_{\Gamma_R} g_z(\x)\overline{v}(\x) ds
	\quad\forall\ v\in V,
\end{align*}
which can be written equivalently as 
\begin{equation}
	(I-T(z))u_z=G_z\quad\text{in }V,
\end{equation}
where $T(z), G_z:V\rightarrow V$ are defined, respectively, as
\begin{align*}
	\dual{T(z)u}{v}
		&=(1+z^2) \int_D u(\x)\overline{v}(\x) d\x
			 +iz \int_{\Gamma_R} u(\x)\overline{v}(\x) ds
			 \quad\forall\ v\in V,\\
	\dual{G_z}{v}
		&=\int_{\Gamma_R} g_z \overline{v}(\x) ds
		\quad\forall\ v\in V.
\end{align*} 
Therefore, $\mcS(z)=(I-T(z))^{-1}G_z$.
We prove that $T(z)$ is compact in all open bounded connected subsets of the complex plane $\C$.
We write $T(z)$ as $T(z)=\widetilde T(z)\circ J$, where $J$ is the compact embedding $J:V\rightarrow H^{1/2+\varepsilon}(D)$, and $\widetilde T(z):H^{1/2+\varepsilon}(D)\rightarrow V$. Hence, in order to prove the compactness of $T(z)$, it is enough to show that $\widetilde T(z)$ is continuous. For all $u\in H^{1/2+\varepsilon}(D)$, we have
\begin{align}
	\nonumber
	\norm{\widetilde T(z)u}{V}
	& = \sup_{v\in V, \norm{v}{V}=1} \abs{\dual{\widetilde T(z)u}{v}}\\
	\nonumber
	& = \sup_{v\in V, \norm{v}{V}=1} 
		\abs{ (1+z^2) \int_D u(\x)\overline{v}(\x) d\x
		 +iz \int_{\Gamma_R} u(\x)\overline{v}(\x) ds  }	\\
	\nonumber
	&\leq \sup_{\norm{v}{V}=1} 
		\left( \abs{1+z^2}\norm{u}{L^2(D)}\norm{v}{L^2(D)}
		 +\abs{z}\norm{u}{L^2(\Gamma_R)}\norm{v}{L^2(\Gamma_R)}\right)\\
	\nonumber
	&\leq \sup_{\norm{v}{V}=1} 
		\left( \abs{1+z^2}\norm{u}{L^2(D)}\norm{v}{L^2(D)}
		 +\abs{z}\norm{u}{L^2(\partial D)}\norm{v}{L^2(\partial D)}\right)\\
	\nonumber
	&\leq \sup_{\norm{v}{V}=1} 
		\left( \abs{1+z^2}\norm{u}{L^2(D)}\norm{v}{L^2(D)}
		 +C_{tr}^2 \abs{z}\norm{u}{H^{1/2+\varepsilon}(D)}
		 	\norm{v}{H^{1/2+\varepsilon}(D)}\right)\\
	\nonumber
	&\leq \max\{\abs{1+z^2},C_{tr}^2\abs{z}\}\norm{u}{H^{1/2+\varepsilon}(D)},
\end{align}
where $C_{tr}$ is the continuity constant of the trace operator $\gamma:H^{1/2+\varepsilon}(D)\rightarrow L^2(\partial D)$ (see, e.g.,~\cite[Theorem 5.36]{Adams2003}). 
Applying Theorem~\ref{thm:steinberg}, we conclude that $(I-T(z))^{-1}$ is meromorphic in all open bounded and connected subsets of $\C$ and, since $G_z$ is linear in $z$ (hence holomorphic in $\C$), the same conclusion applies to the frequency response function $\mcS(z)=(I-T(z))^{-1}G_z$.
Moreover, since Theorem~\ref{thm:scattering_wp} states that $\mcS$ is well defined in $\C^+$, we deduce that all poles of $\mcS$ must have negative imaginary part. 
\end{proof}

\subsection{LS-Pad\'e approximant of the frequency response map}

We construct the LS-Pad\'e approximant of the frequency response map $\mcS$ following Algorithm~\ref{al:pade}. Having fixed $z_0\in\C^+$, $N,\ M$, and $E\geq M+N$, the coefficients of the denominator are computed by identifying the eigenvector corresponding to the smallest eigenvalue of the matrix~\eqref{eq:gram_matrix}, where the $\beta$-th Taylor coefficient of $\mcS$, $\coeff{\mcS}{\beta,z_0}$, solves the following recursive problem:
\begin{align}
	\nonumber
	&\int_D\Grad \coeff{\mcS}{\beta,z_0}(\x)\cdot\overline{\Grad v}(\x) d\x
	-z_0^2 \int_D \coeff{\mcS}{\beta,z_0}(\x)\overline{v}(\x) d\x
	-iz_0 \int_{\Gamma_R} \coeff{\mcS}{\beta,z_0}(\x)\overline{v}(\x) ds\\
	\nonumber
	&\quad=2 z_0 \int_D \coeff{\mcS(\x)}{\beta-1,z_0}\overline{v}(\x) d\x
	+i \int_{\Gamma_R} \coeff{\mcS(\x)}{\beta-1,z_0}\overline{v}(\x) ds\\
	\label{eq:scattering_tay}
	&\quad+\int_D \coeff{\mcS(\x)}{\beta-2,z_0}\overline{v}(\x) d\x
	+\frac{1}{\beta !}\int_{\Gamma_R} \frac{d^\beta}{dz}g_z(\x)|_{z=z_0}
		\cdot\overline{v}(\x) ds
	\quad\forall\ v\in V.
\end{align}
Since the PDE operator in~\eqref{eq:scattering_tay} is the same as in~\eqref{eq:total_field_weak}, and the linear form at the right-hand side is bounded, by applying Theorem~\ref{thm:scattering_wp}, we conclude that problem~\eqref{eq:scattering_tay} is well-posed for any $z\in\C^+$.

\begin{figure}[h]
	\centering
	\includegraphics[width=0.3\textwidth]{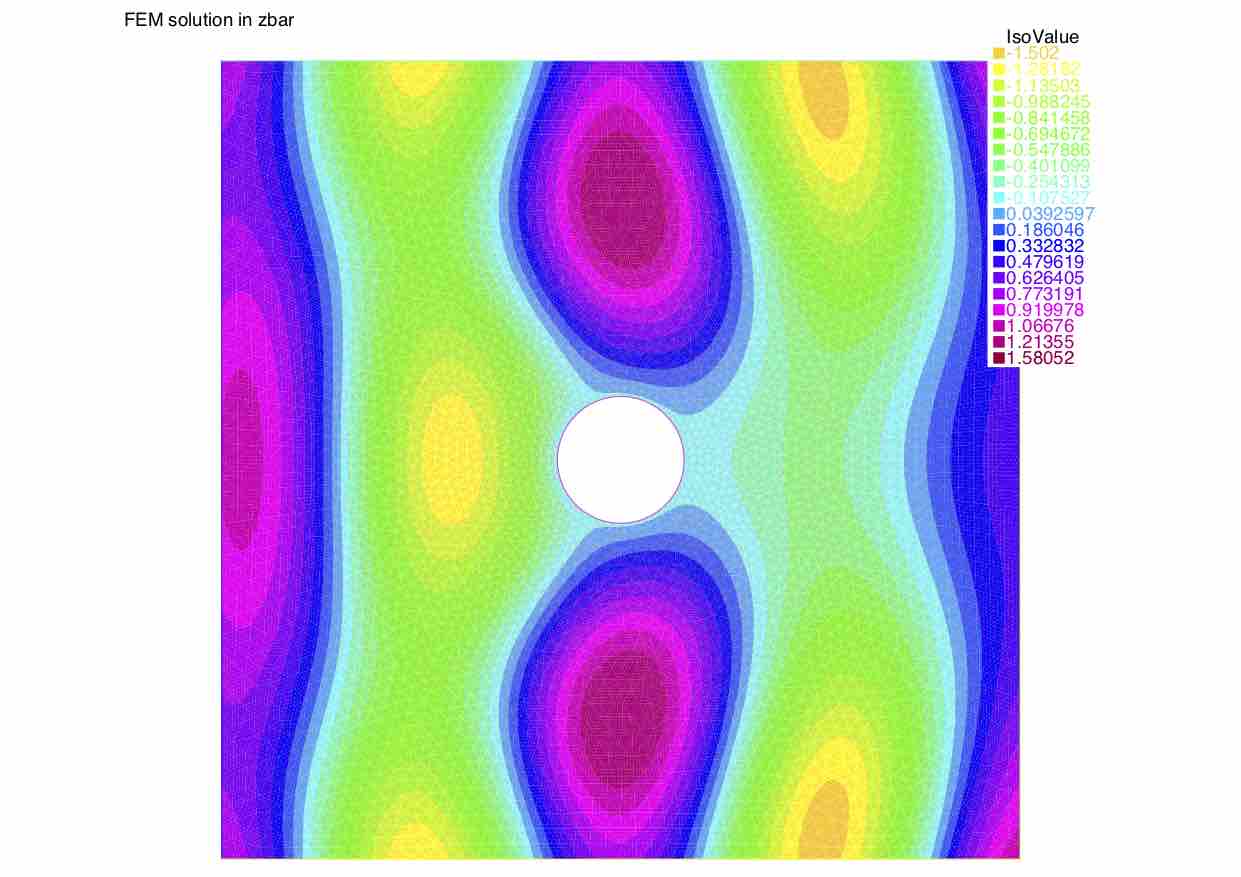}
	\includegraphics[width=0.3\textwidth]{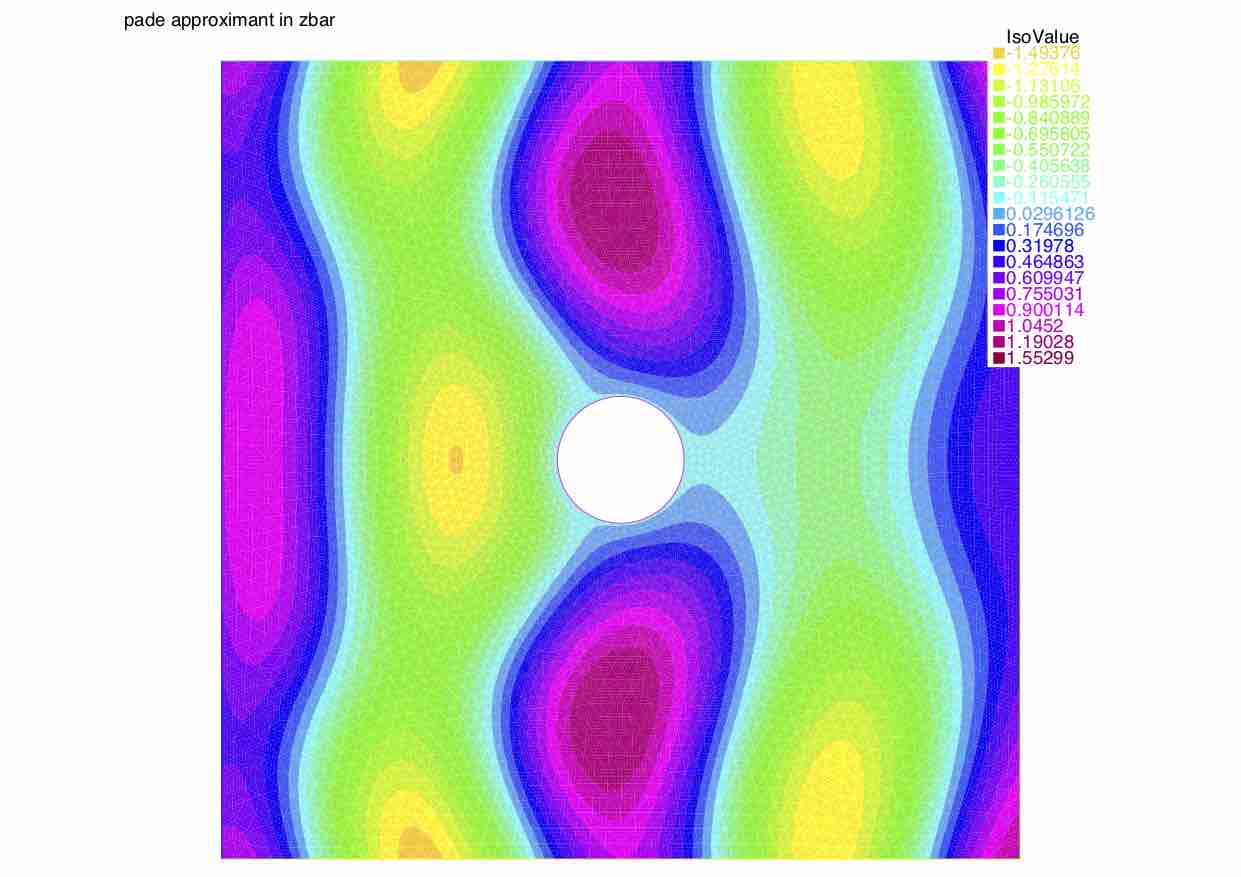}
	\includegraphics[width=0.3\textwidth]{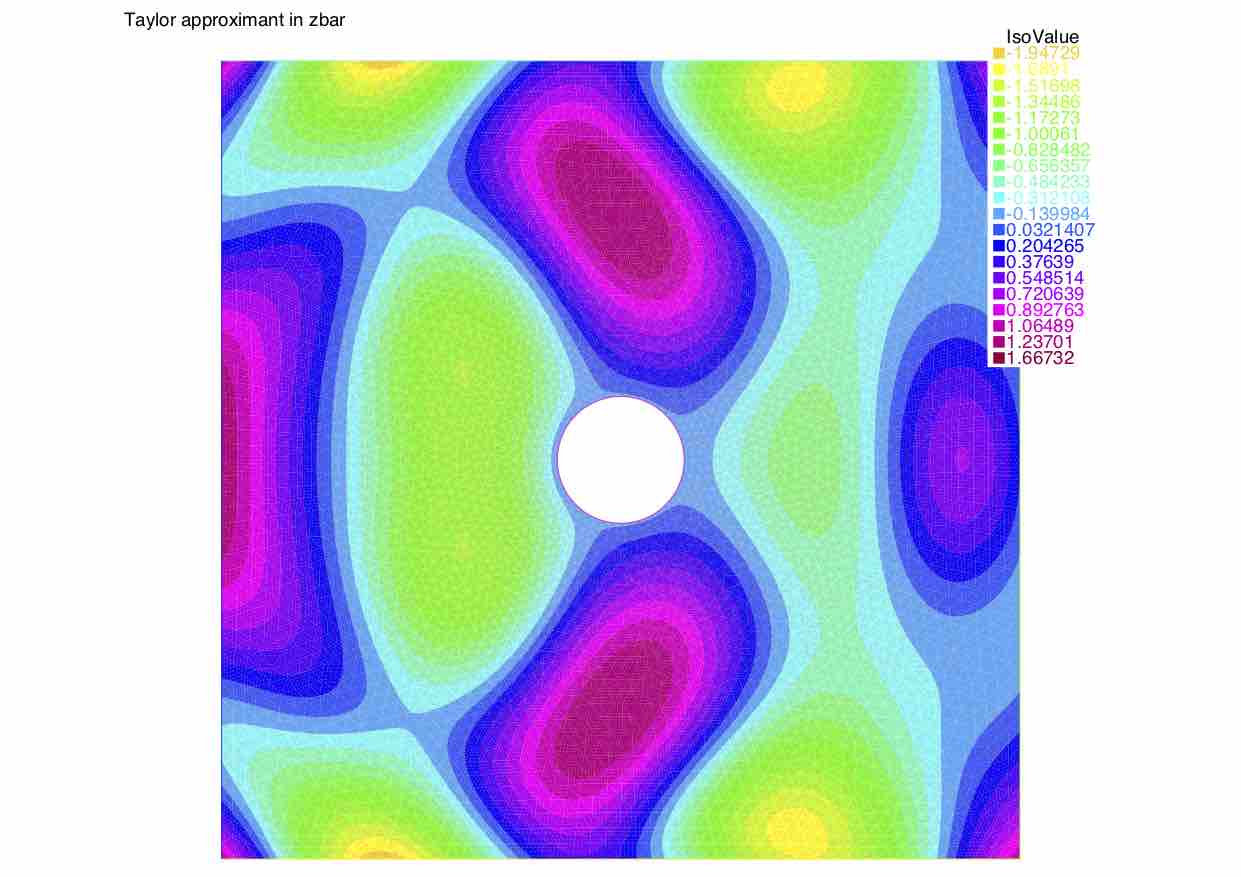}
	\caption{Comparison between the $\mathbb P^3$ finite element solution of problem~\eqref{eq:total_field_weak} (left), with its LS-Pad\'e approximation $\pade{\mcS}(z)$ centered in $z_0=3+0.5i$ with degree $(M,N)=(10,2)$ (center), and its Taylor polynomial centered in $z_0=3+0.5i$ with degree $M+N=12$ (right), in the point $z=2$.}
	\label{fig:scattering_comparison}
\end{figure}

Let $u^i$ be the incident wave traveling along the direction $\d=(\cos(0),\sin(0))$ with wavenumber $z=2$. Figure~\ref{fig:scattering_comparison} (left) represents the solution of problem~\eqref{eq:total_field_weak} computed via the finite element method with polynomials of degree 3. Figure~\ref{fig:scattering_comparison} (center) and (right) represents the LS-Pad\'e approximation $\pade{\mcS}(z)$ at $z=2$ with center $z_0=3+0.5i$ and degree $(M,N)=(10,2)$ (and parameters $\rho=\abs{z-z_0}$, $E=M+N$), and the Taylor polynomial centered in $z_0=3+0.5i$ with degree $E=12$, respectively. Both the Pad\'e and the Taylor approximant are constructed starting from the set of evaluations $\{\mcS(z_0),\coeff{\mcS}{1}(z_0),\ldots,\coeff{\mcS}{12}(z_0)\}$. The LS-Pad\'e approximant reproduces the behavior of the reference solution much better than the Taylor approximant, and the LS-Pad\'e relative error in the weighted $H^1(D)$-norm $err_{pade}=0.101089$ is much smaller than the Taylor one $err_{tay}=0.611428$.

Let $z=3$, and $z_0=3+0.5i$. In Figure~\ref{fig:scattering_err_vs_M_z3} (left) we plot the relative LS-Pad\'e approximation error versus the degree of the LS-Pad\'e numerator, for different values of denominator degree. In Figure~\ref{fig:scattering_err_vs_M_z3} (right), the relative error obtained by approximating the frequency response map with the Taylor polynomial (black dashed line), and with the LS-Pad\'e approximant are compared. Also the diagonal LS-Pad\'e approximant is considered (dashed purple line), where the LS-Pad\'e numerator and denominator have the same degree. 
In Figure~\ref{fig:scattering_err_vs_M_z3} (right), the errors are plotted versus the number of derivatives $\coeff{\mcS}{\beta,z_0}$, $\beta=0,\ldots, E$ computed (i.e., the number of PDEs solved offline). Since $\ball{z_0}{\abs{z-z_0}}$, the disk with center $z_0=3+0.5i$ and radius $r_1=\abs{z-z_0}=0.5$, is contained in the half plane where the frequency response map is holomorphic (see Proposition~\ref{prop:scattering_mer}), the Taylor series centered in $z_0$ converges, and the Taylor approximation error is comparable to the LS-Pad\'e approximation error. Figure~\ref{fig:scattering_err_vs_M_z45} presents analogous plots as in Figure~\ref{fig:scattering_err_vs_M_z3}, for the point $z=2$. In this case, $\ball{z_0}{\abs{z-z_0}}\cap\{\Imag{z}<0\}\neq\emptyset$, and the Taylor series diverges.

\begin{figure}[h]
	\centering
	\includegraphics[width=0.45\textwidth]{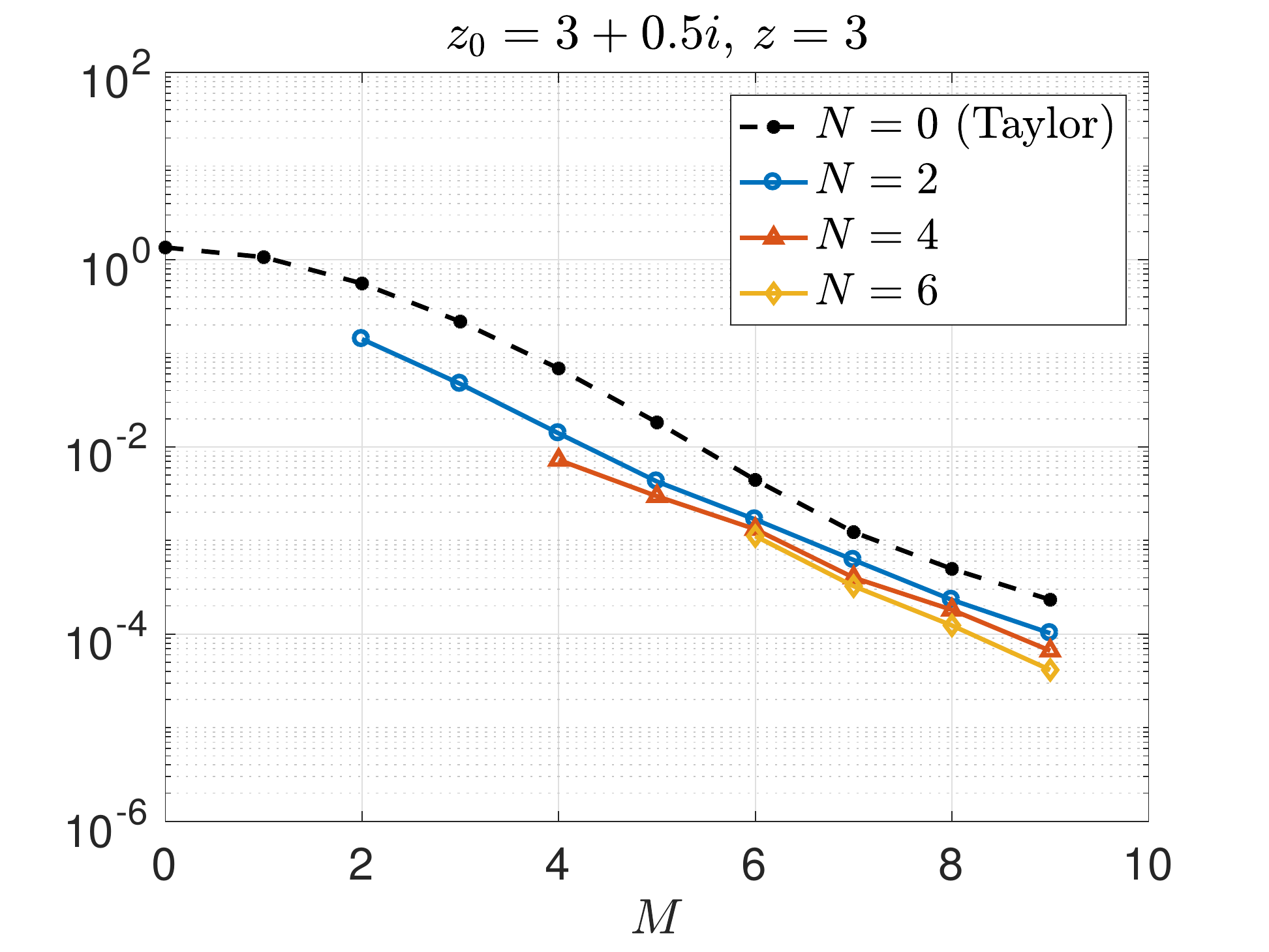}
	\includegraphics[width=0.45\textwidth]{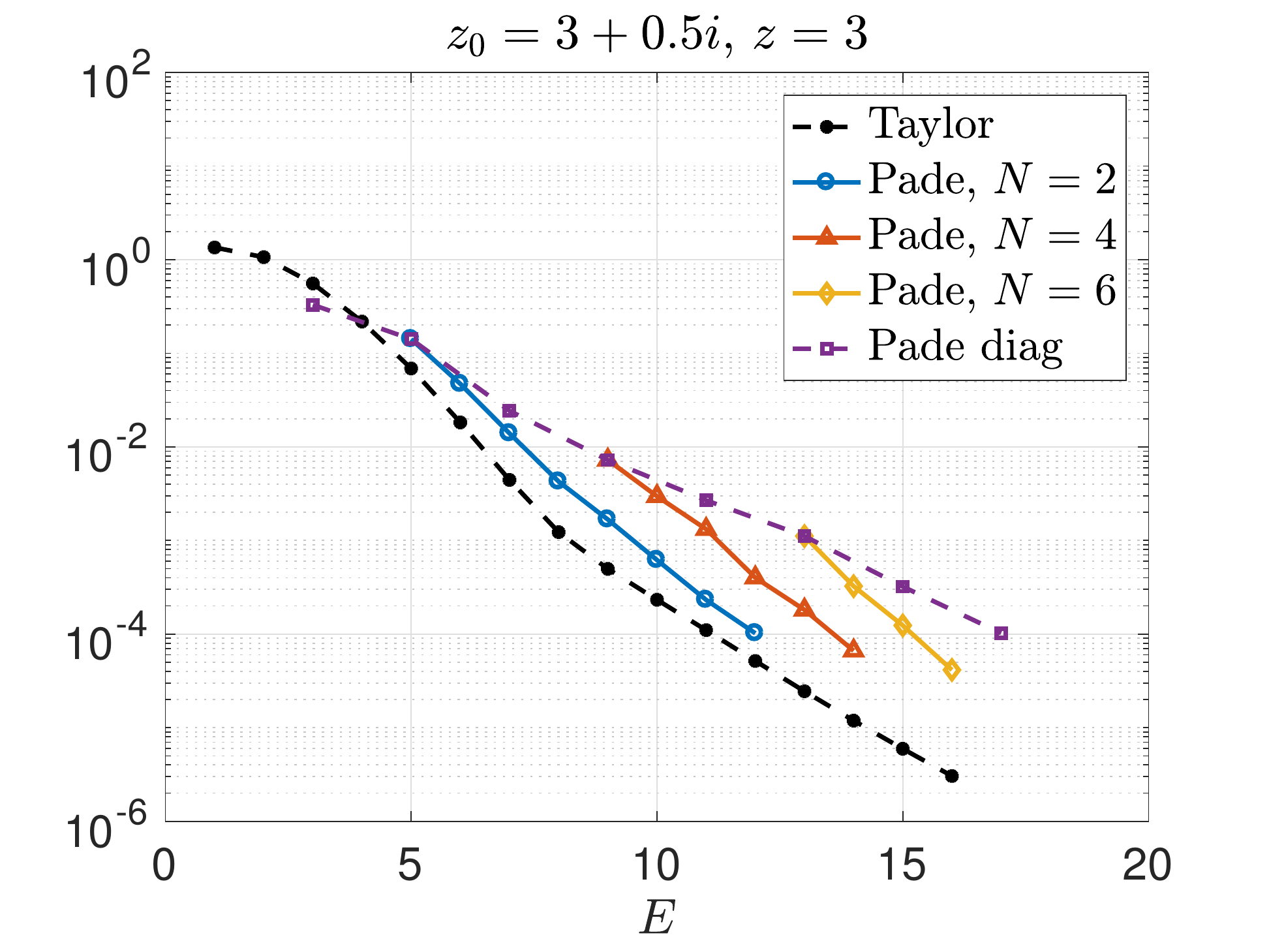}
	\caption{Relative Taylor and LS-Pad\'e approximation error in $z=3$
	plotted versus the numerator degree (left), and the number of derivatives $\coeff{\mcS}{\beta,z_0}$, $\beta=0,\ldots, E$ computed offline (right).}
	\label{fig:scattering_err_vs_M_z3}
\end{figure}

\begin{figure}[h] 
	\centering
	\includegraphics[width=0.45\textwidth]{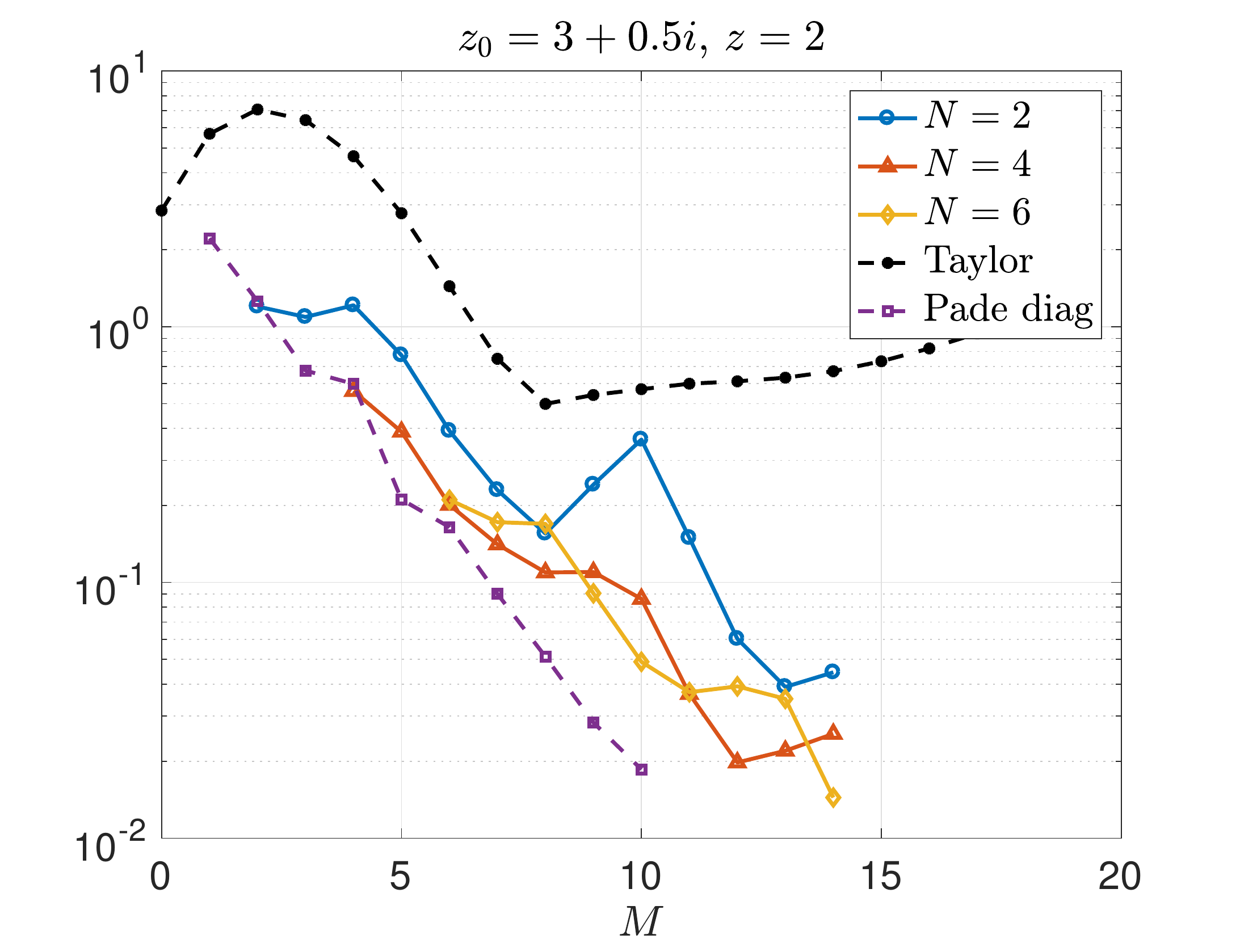}
	\includegraphics[width=0.45\textwidth]{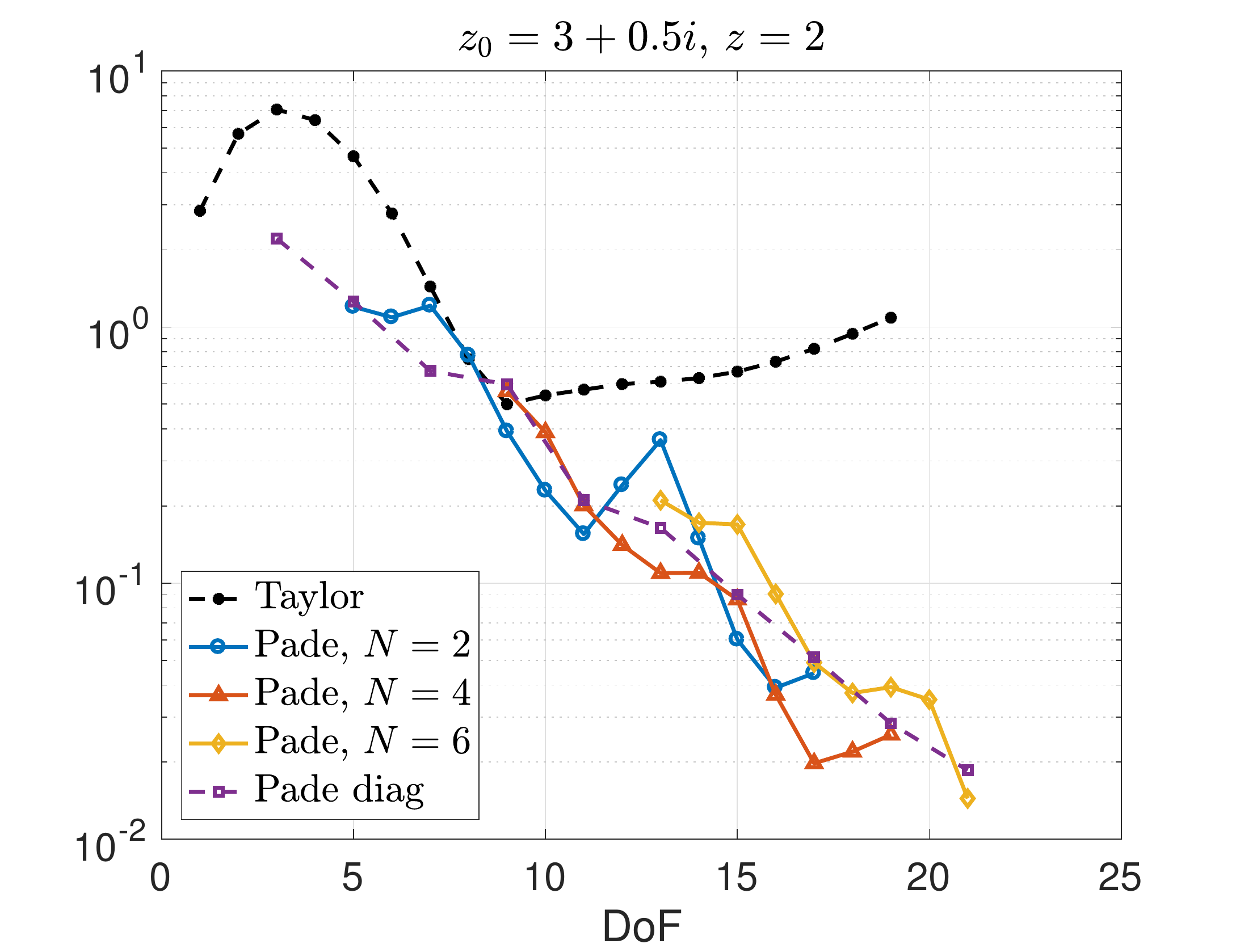}
	\caption{Relative Taylor and LS-Pad\'e approximation error in $z=2$
	plotted versus the numerator degree (left), and the number of derivatives $\coeff{\mcS}{\beta,z_0}$, $\beta=0,\ldots, E$ computed offline (right).}
	\label{fig:scattering_err_vs_M_z45}  
\end{figure}


\section{Application in high frequency regime}
\label{sec:high_frequency}

In this section, we want to study the approximation properties of the LS-Pad\'e approximant in the high frequency regime. As in~\cite{Bonizzoni2016}, we consider Problem~\eqref{prob:model_problem_parametric} with $D=(0,\pi)\times (0,\pi)$, $\Gamma_D=\partial D$, $g_D=0$, and $f(\x)=-\Delta w(x)-\nu^2 w(x)$, where $w(x)$ is the product between the plane wave $\e^{-i\nu\d\cdot\x}$ with wavenumber $\nu=\sqrt{51}$ traveling along the direction $\d=(\cos(\pi/6),\sin(pi/6))$ and the normalized quadratic bubble function vanishing on $\partial D$. Note that $w(x)$ is the exact solution of the Helmholtz equation~\eqref{eq:helmholtz_weak}, i.e., $\mcS(z)=w$, when $z=\nu^2$. The interval of frequencies we are interested in is $K=[39,55]$, which contains 6 eigenvalues of the Dirichlet-Laplace operator: $40, 41, 45, 50, 52, 53$.

In Figure~\ref{fig:high_frequency_FEM_zbar}, we plot the numerical solution $\mcS_h\in V$ of problem~\eqref{eq:helmholtz_weak} with $z=51$, computed via $\mathbb P^3$ continuous finite elements. Observe that the relative finite element error is of the order of $10^{-5}$. In Figure~\ref{fig:high_frequency_pade_zbar}, the LS-Pad\'e approximant $\mcS_{h,P}$ centered in $z_0=47+0.5i$ evaluated in $z=51$ is represented for two different values of the denominator degree. Due to the fact that more derivatives are employed in the right plot, more accurate results are obtained with higher denominator polynomial degrees.

In Figure~\ref{fig:high_frequency_error_zbar}, we plot the LS-Pad\'e approximation error w.r.t. the exact solution, in $z=51$, for different values of the degree of the denominator, and we compare it with the numerical rate~\eqref{eq:numerical_rate}. When $N=2,4$, the LS-Pad\'e technique works as expected (or even better), whereas for $N=6$ the error is no longer decreasing. We believe that this behavior is caused by the ill-conditioning of Step 7 in Algorithm~\ref{al:pade}), i.e., the computation of the (normalized) eigenvector of the Gramian matrix $\G$ defined in~\eqref{eq:gram_matrix}.

We partition uniformly the interval of interest $K$ in $100$ subintervals. At each point $z$ of the grid we have computed the numerical solution $\mcS_h(z)$ of the Helmholtz problem~\eqref{eq:helmholtz_weak}, and the LS-Pad\'e approximant $\mcS_{h,P}$ (see Figure~\ref{fig:high_frequency_norm_comparison}), as well as the relative error $\frac{\normw{\mcS_h(z)-\mcS_{h,P}(z)}{H^1(D)}{\sqrt{\Real{z_0}}}}{
\normw{\mcS_h(z)}{H^1(D)}{\sqrt{\Real{z_0}}}}$ (see Figure~\ref{fig:high_frequency_relative_error}).
In Figure~\ref{fig:high_frequency_roots}, we study the convergence of the roots of the LS-Pad\'e denominator $\den$ to the exact Laplace eigenvalues. 
For all degrees of the LS-Pad\'e denominator $\den$, there are two roots of $\den$ which converge to the two Laplace eigenvalues closest to $z_0$. Concerning the other roots, we observe two regimes: the error decreases for $M$ smaller than a fixed value $M^\star$ which depends on $N$ ($M^\star=12$ if $N=4$, $M^\star=9$ if $N=6$); for $M>M^\star$, the problem becomes ill-conditioned and the roots do not converge anymore to the Laplace eigenvalues. This behavior explains also the reason why in Figure~\ref{fig:high_frequency_relative_error} only 4 peaks are identified by the LS-Pad\'e approximant.
The ill-conditioning of the eigenvalue problem limits the applicability of the method, especially in high frequency regime, where the singularities are dense. To overcome this problem, we are currently investigating the multi-point generalization of the single-point LS-Pad\'e method proposed in this paper. In a multi-point framework, the number of derivatives to be computed are split over the set of centers. In particular, instead of computing $M+N$ derivatives in a single center $z_0$, $\lceil(M+N)/n\rceil$ derivatives will be computed in each center $z_i$, for $i=0,\ldots,n-1$.

\begin{figure}[htb]
\centering
\includegraphics[width=0.45\textwidth]{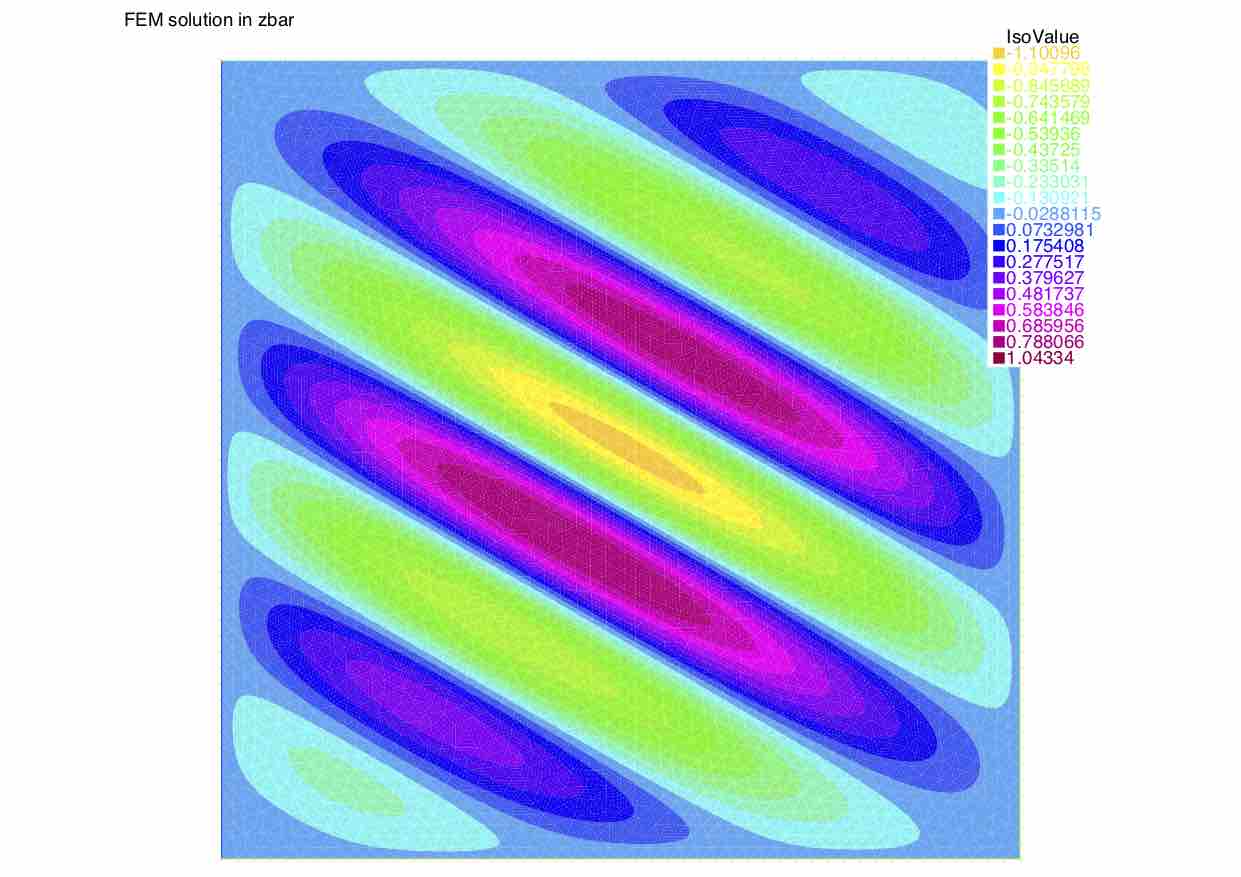}
\caption{Finite element solution $\mcS_h(51)\in V$ of problem~\eqref{eq:helmholtz_weak}}
\label{fig:high_frequency_FEM_zbar}
\end{figure}

\begin{figure}[htb]
\centering
\includegraphics[width=0.45\textwidth]{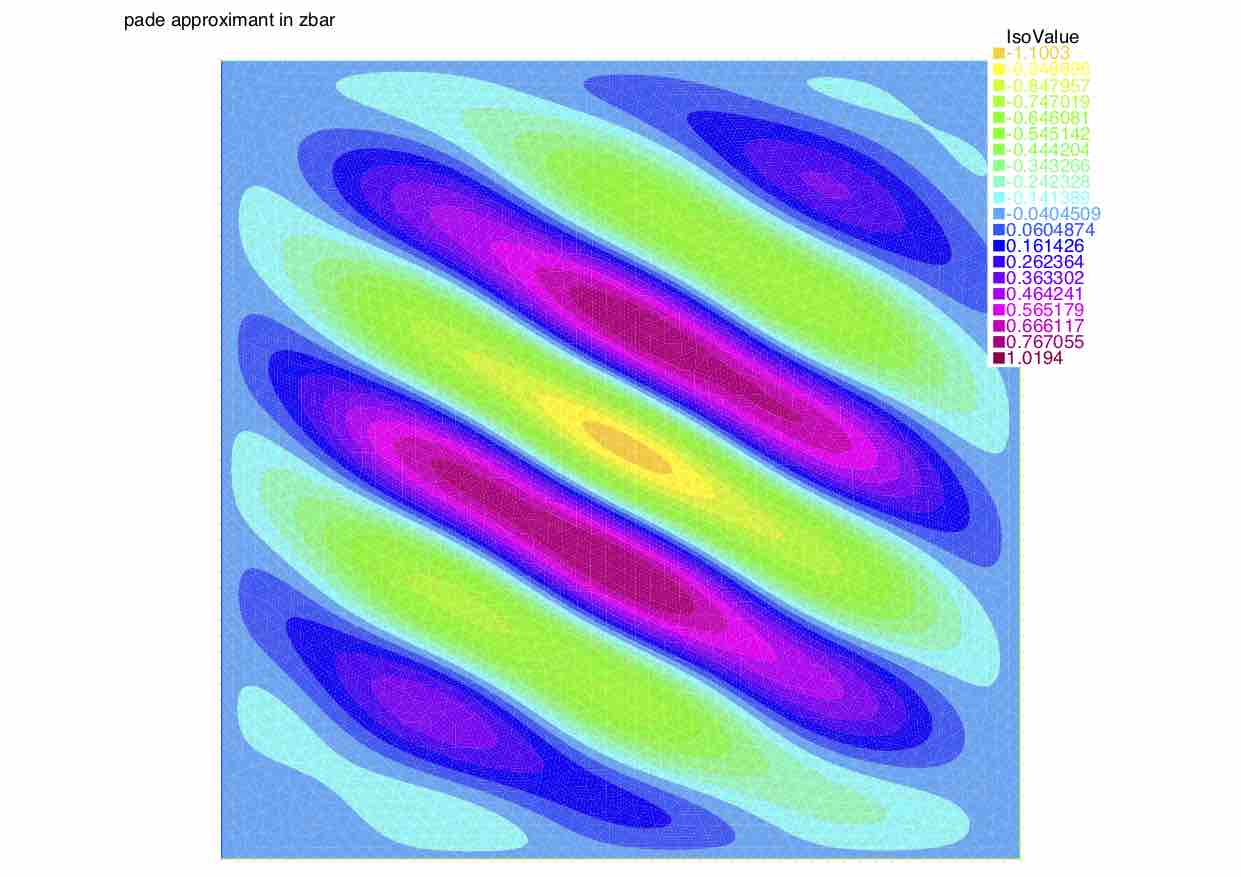}
\includegraphics[width=0.45\textwidth]{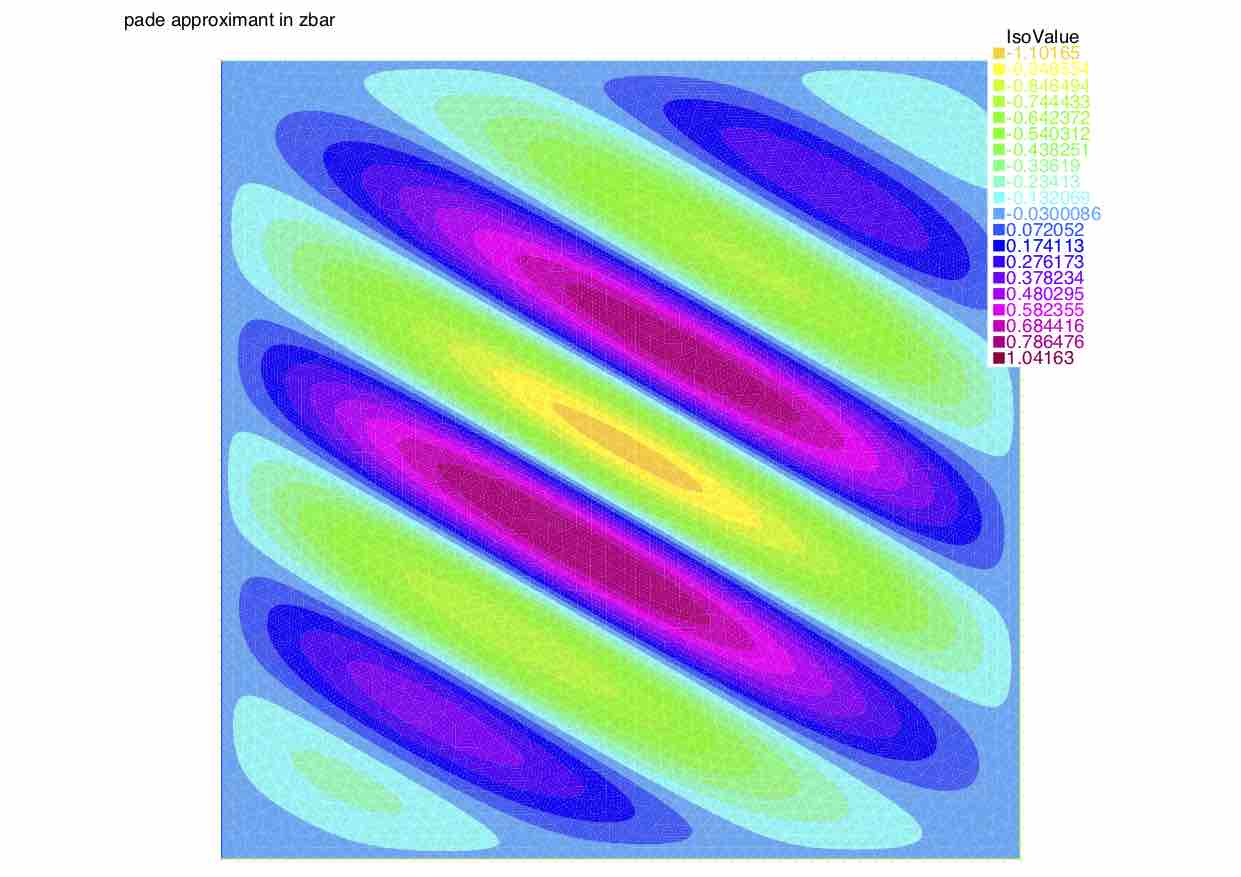}
\caption{LS-Pad\'e approximant $\mcS_{h,P}(51)$ centered in $z_0=47+0.5i$, with degrees $M=10$ and $N=2$ (left), $N=4$ (right).}
\label{fig:high_frequency_pade_zbar}
\end{figure}

\begin{figure}[htb]
\centering
\includegraphics[width=0.5\textwidth]{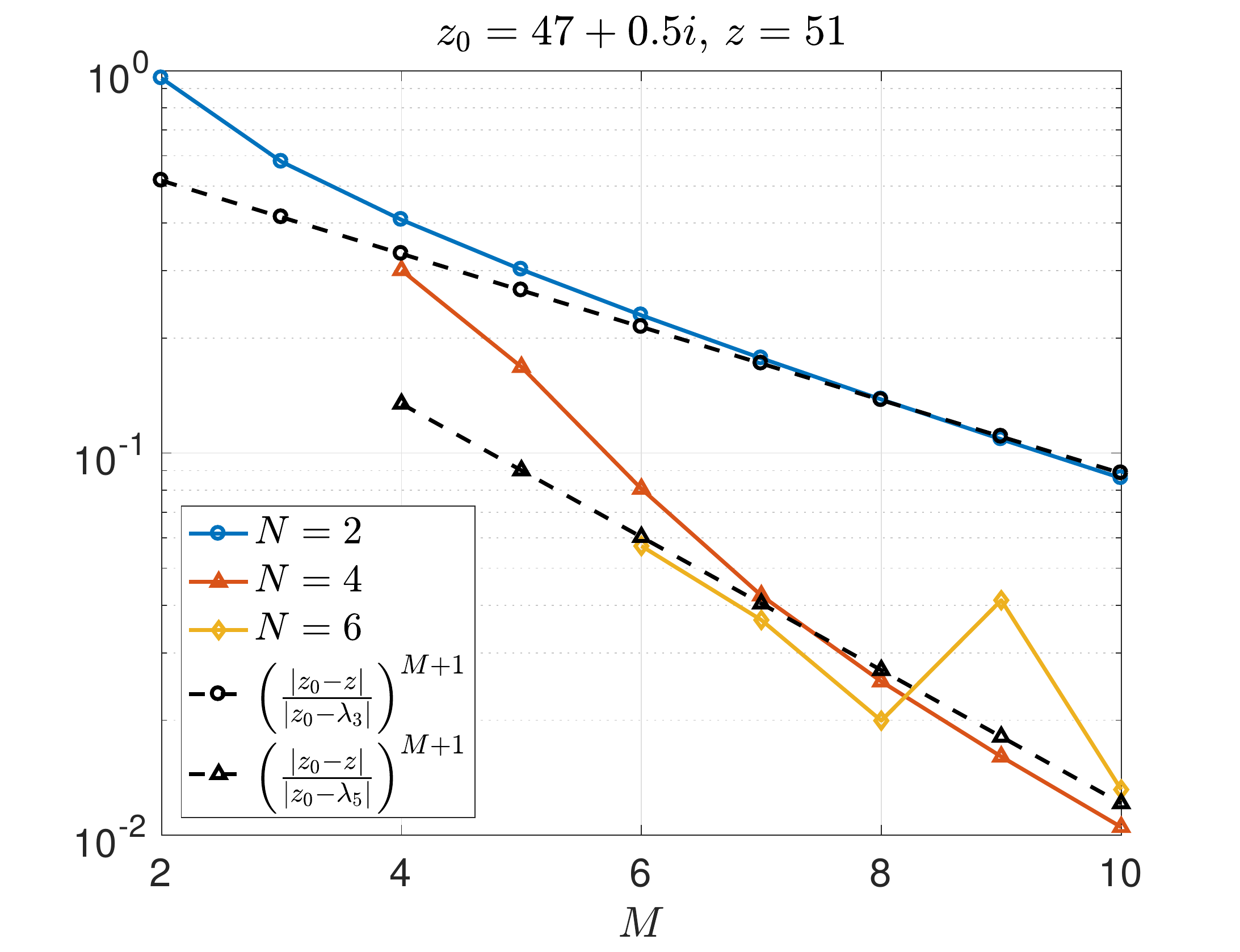}
\caption{Relative error 
$\frac{\normw{\mcS(z)-\mcS_{h,P}(z)}{V}{\sqrt{\Real{z_0}}}}{\normw{\mcS(z)}{V}{\sqrt{\Real{z_0}}}}$ in the point $z=51$, as a function of the degree of the numerator $M$.}
\label{fig:high_frequency_error_zbar}
\end{figure}

\begin{figure}[htb] 
\centering
\includegraphics[width=0.45\textwidth]{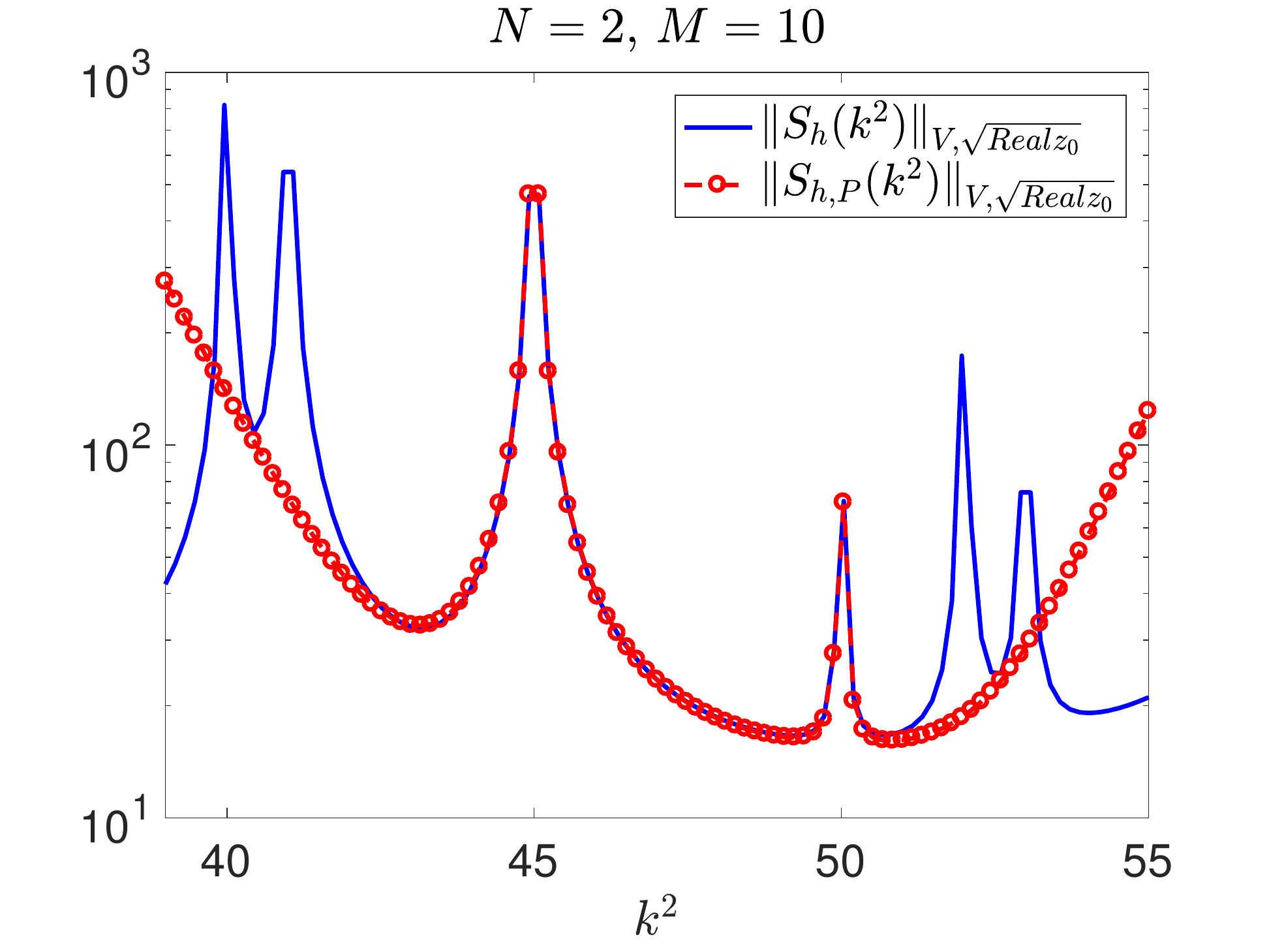}
\includegraphics[width=0.45\textwidth]{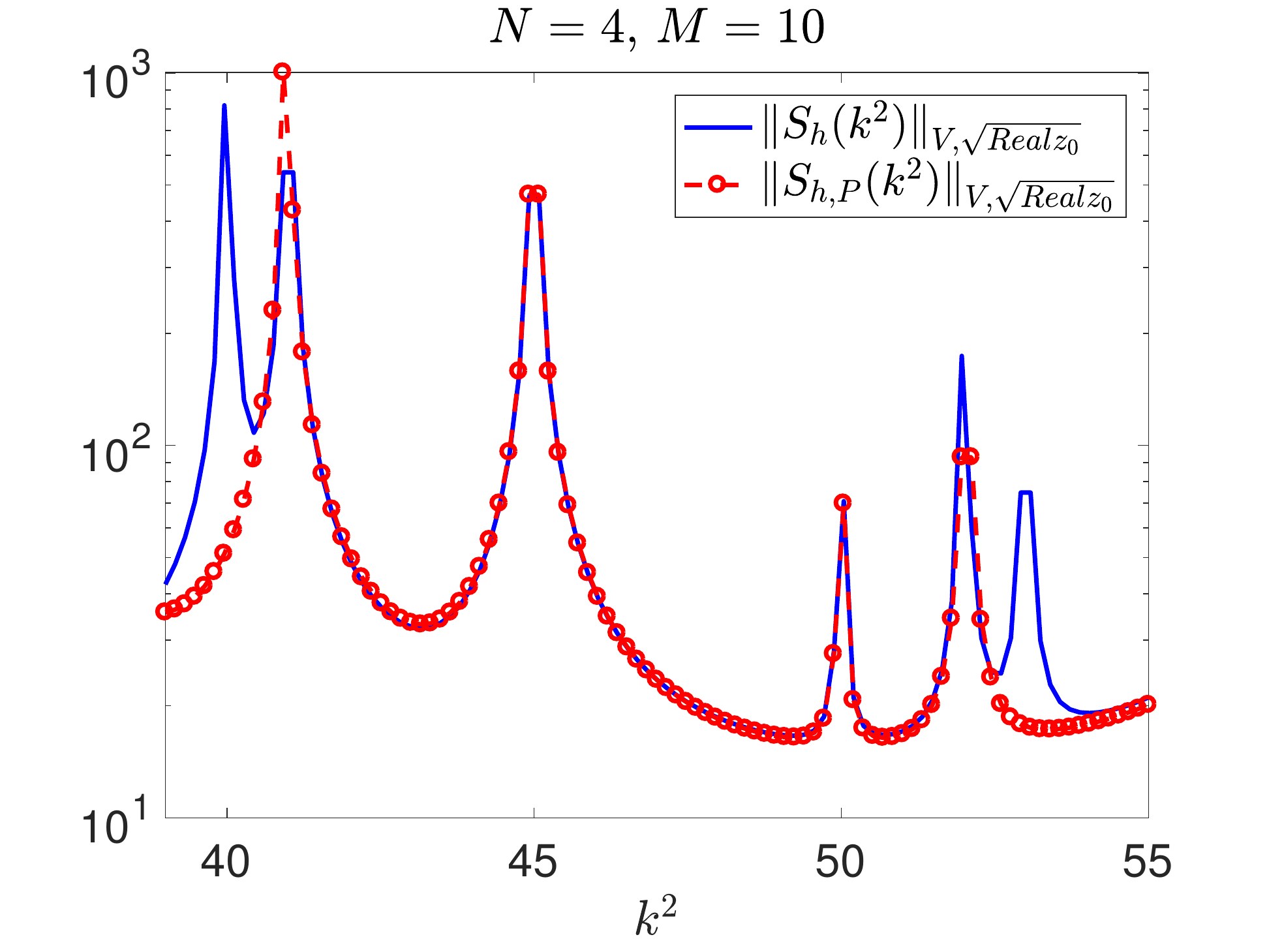}
\includegraphics[width=0.45\textwidth]{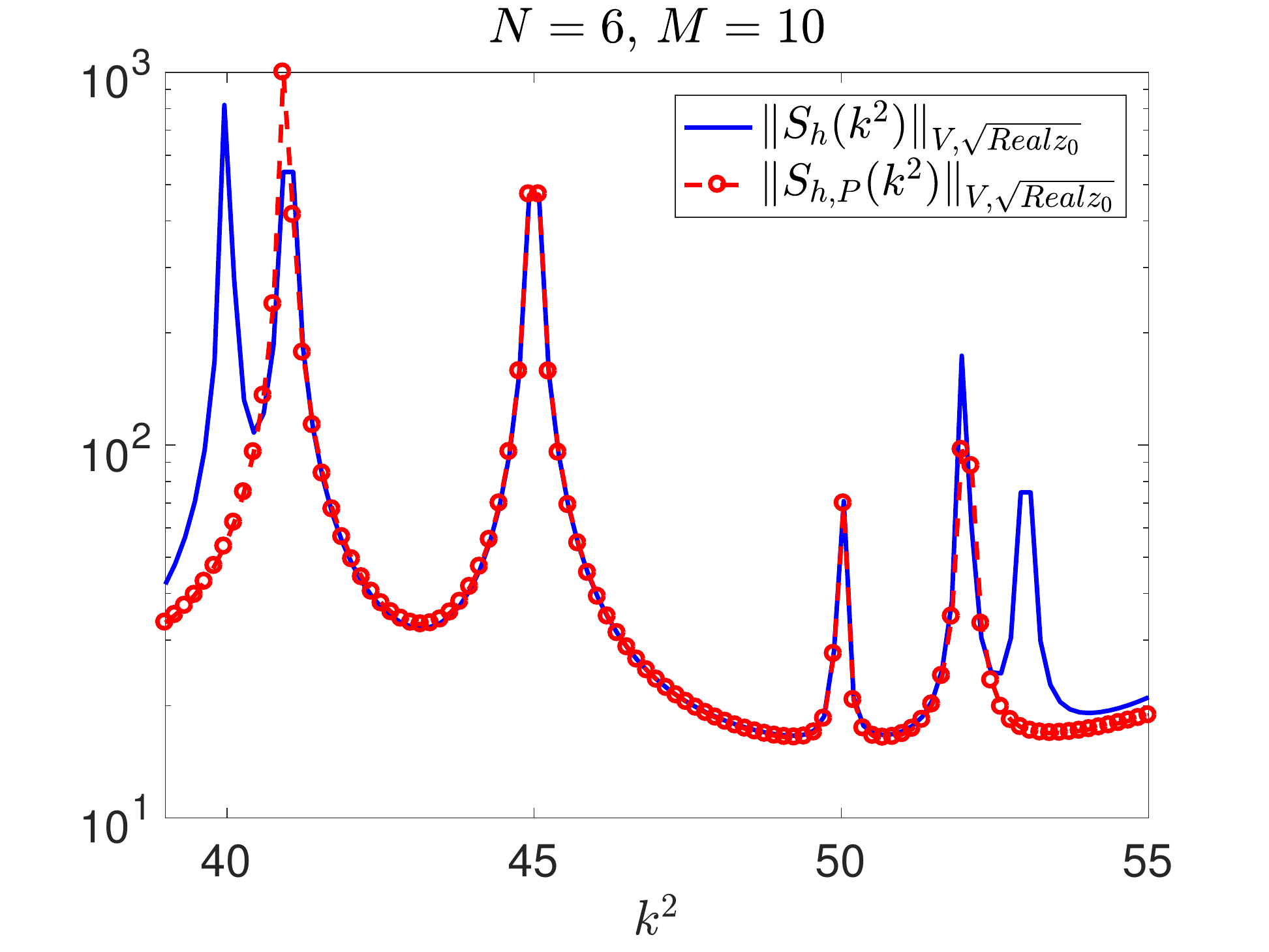}
\caption{Comparison between the weighted $H^1(D)$-norm of the $\mathbb P^3$ finite element solution $\mcS_h(z)$, and the weighted $H^1(D)$-norm of the LS-Pad\'e approximant $\pade{\mcS}(z)$, for numerator degree $M=10$ and denominator degree $N=2$, $N=4$ and $N=6$.}
\label{fig:high_frequency_norm_comparison}
\end{figure}

\begin{figure}[htb]
\centering
\includegraphics[width=0.5\textwidth]{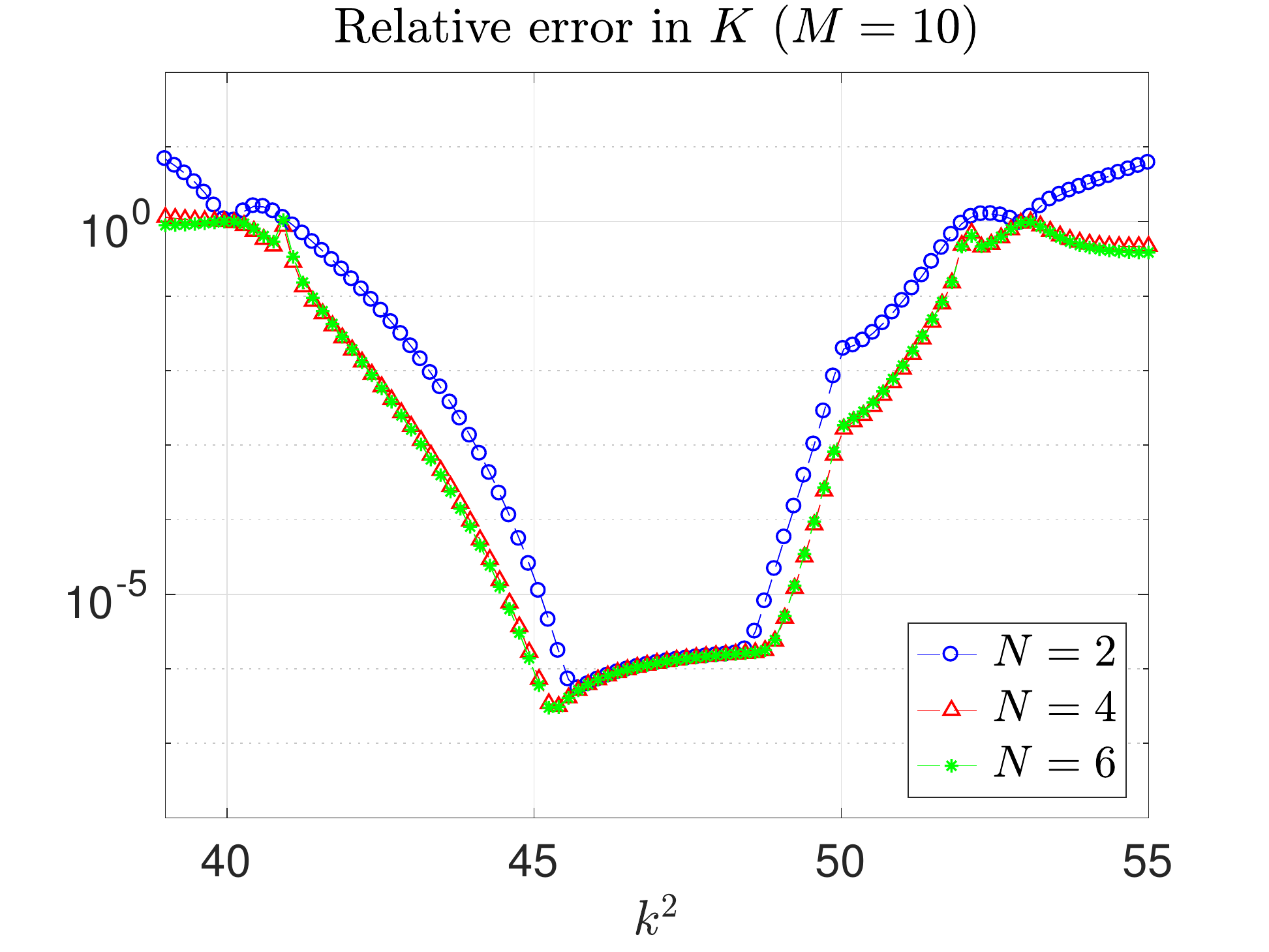}
\caption{Relative error $\frac{\normw{\mcS_h(z)-\mcS_{h,P}(z)}{H^1(D)}{\sqrt{\Real{z_0}}}}{
\normw{\mcS_h(z)}{H^1(D)}{\sqrt{\Real{z_0}}}}$ for different values of the degree of the denominator.}
\label{fig:high_frequency_relative_error}
\end{figure}

\begin{figure}[htb]
\centering
\includegraphics[width=0.45\textwidth]{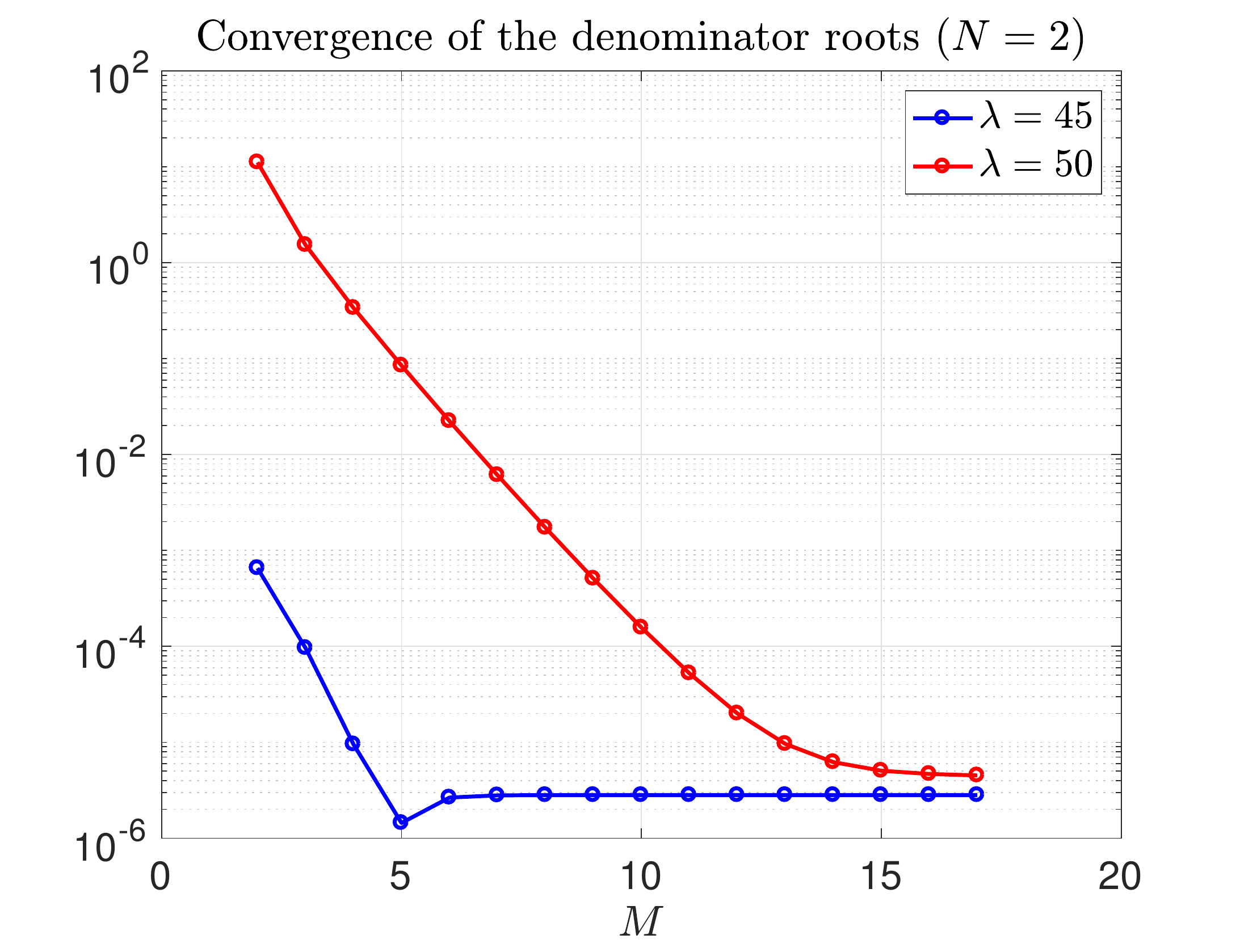}
\includegraphics[width=0.45\textwidth]{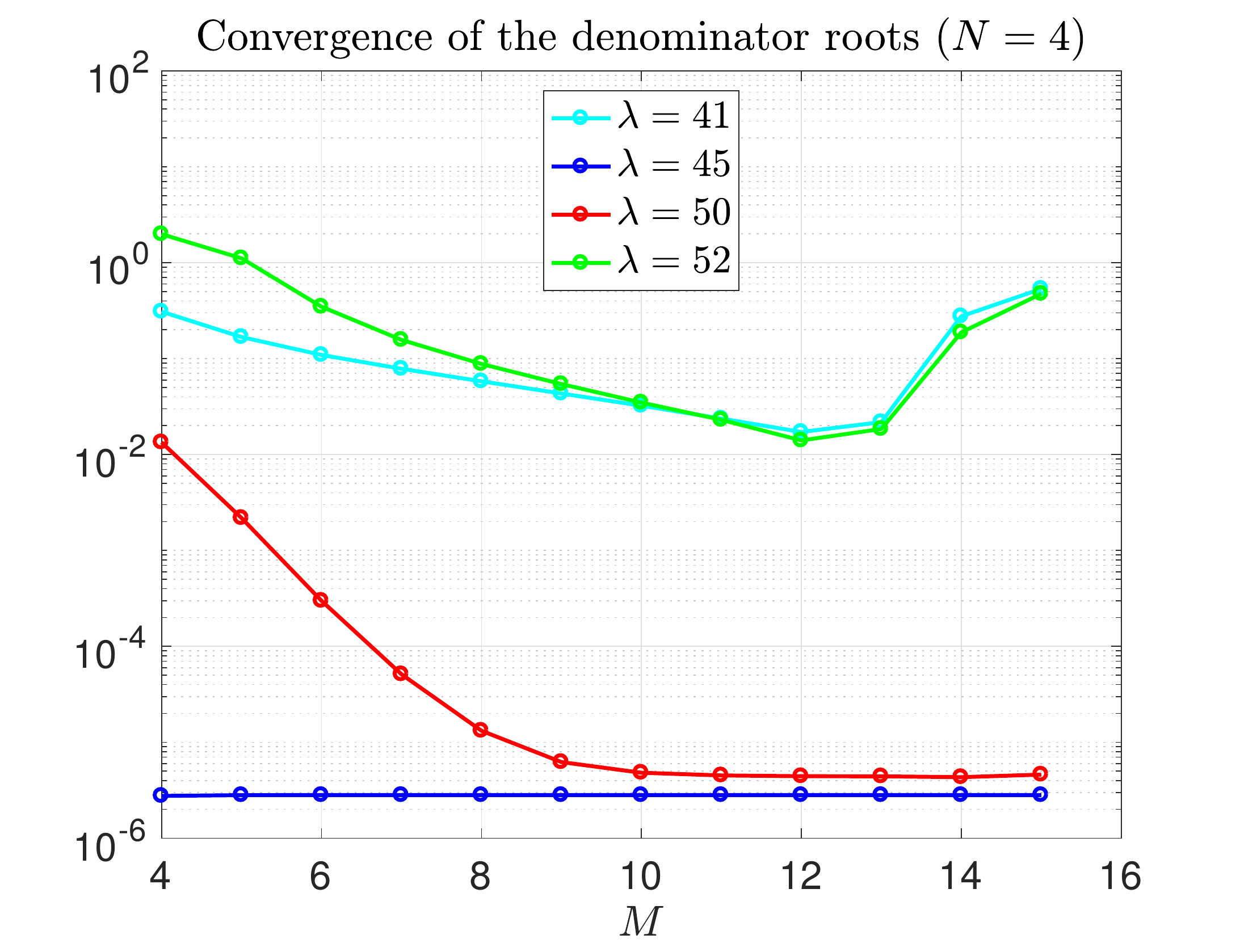}
\includegraphics[width=0.45\textwidth]{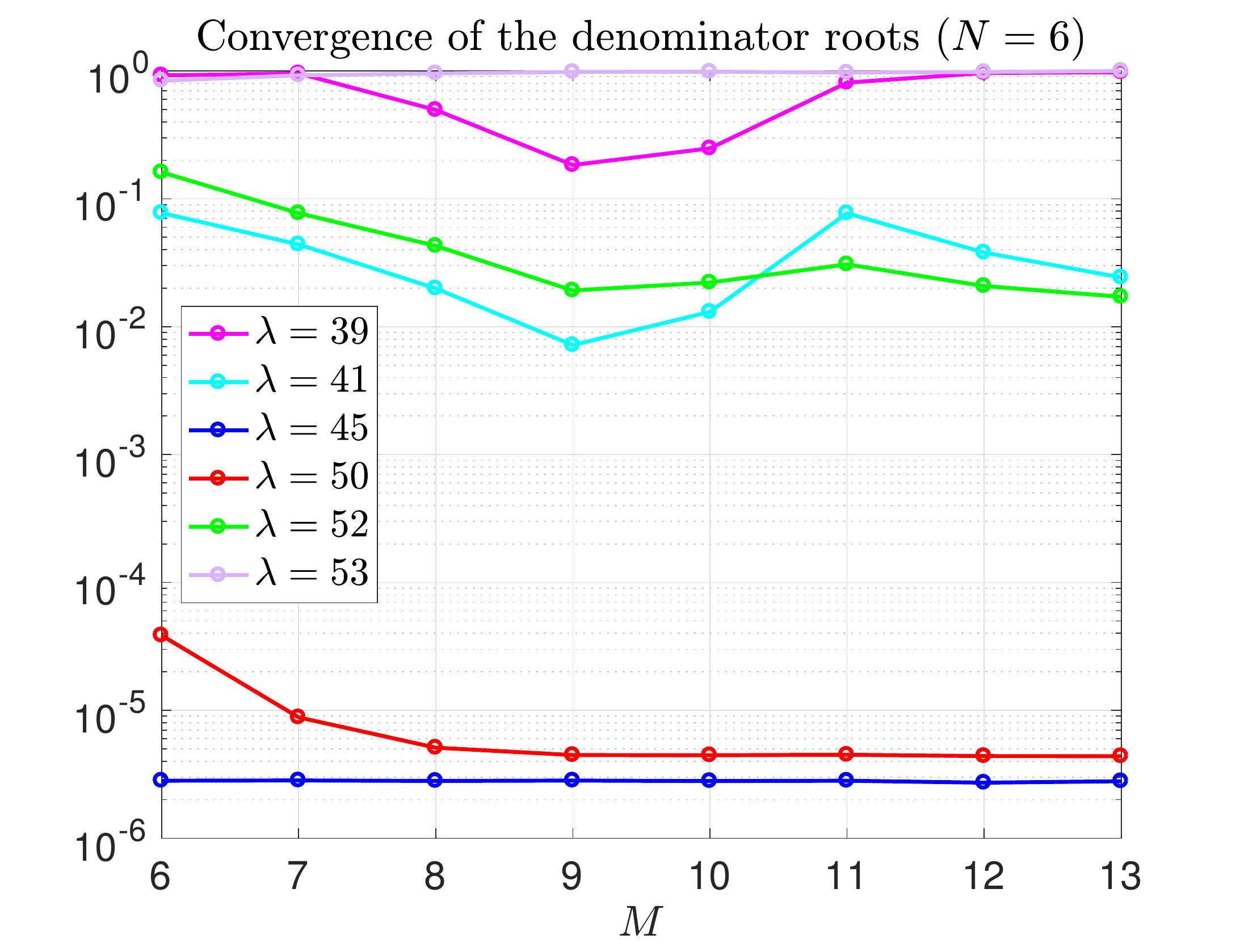}
\caption{Convergence of the roots $r_Q$ of the LS-Pad\'e denominator to the Laplace eigenvalues. The error $\abs{r_Q-\lambda}$ is plotted for $N=2$, $N=4$ and $N=6$.}
\label{fig:high_frequency_roots}
\end{figure}
%
\section{LS-Pad\'e approximant of the stochastic model problem}
\label{sec:stochastic_helmholtz}

This section deals with the stochastic counterpart of Problem~\ref{prob:model_problem_parametric}:

\begin{problem}[Stochastic Model Problem]
	\label{prob:model_problem_stochastic}
	The wavenumber $k^2$ of the Helm\-holtz equation is modeled as a random variable with bounded density function $\mathscr{F}_{k^2}$. In this section, either Dirichlet or Neumann or mixed Dirichlet/Neumann homogeneous boundary conditions on $\partial D$ are considered.
\end{problem}
We introduce a Lipschitz functional $\mcL:V\rightarrow\R$ representing a quantity of interest of the frequency response map $\mcS$, and we define the following two random variables:
\begin{equation}
	\label{eq:X}
	X:=\mcL(\mcS(k^2))
\end{equation}
and
\begin{equation}
	\label{eq:Xp}
	X_P:=\mcL(\pade{\mcS}(k^2))
\end{equation}
where $\pade{\mcS}=\frac{\num}{\den}$ is the LS-Pad\'e approximant of $\mcS$ centered in $z_0$, with $\Real{z_0}=\frac{k^2_{min}+k^2_{max}}{2}$ and $\Imag{z_0}\neq 0$; this guarantees that $z_0\notin\Lambda$, $\Lambda$ being the set of eigenvalues of the Laplacian, with the considered boundary conditions.
Let 
$\phi_X,\phi_{X_P}:\R\rightarrow \C$ denote the characteristic functions of $X$ and $X_P$, respectively, i.e., $\ds{\phi_X(t):=\mean{\e^{itX}}}$, $\ds{\phi_{X_P}(t):=\mean{\e^{itX_P}}}$.
We are interested in studying the LS-Pad\'e approximation error on the characteristic function, i.e., we aim at proving an a priori bound for 
\begin{equation}
\label{eq:errort}
err_t
=\abs{\phi_X(t)-\phi_{X_P}(t)}
\quad\text{for any } t\in \R.
\end{equation}

\begin{theorem}
\label{th:errt_bound}
Let $\mcL:V\rightarrow\R$ be a Lipschitz functional with Lipschitz constant $L$, and let $X,X_P$ be the random variables defined in~\eqref{eq:X} and~\eqref{eq:Xp}.
Given $\alpha>0$, then it holds
\begin{equation}
\label{eq:errt_bound}
err_t \leq \left(2\abs{K_\alpha} 
	+ \abs{t}\,L\,C\frac{1}{\alpha^3} \left(\frac{\rho}{R}\right)^{M+1} \abs{K}\right)
	\sup_{x\in K} \mathscr{F}_{k^2}(x)
	\quad \forall\ t\in\R,
\end{equation}
with the same definitions of $R$, $\rho$, and $K_\alpha$, and the same characterization of $C>0$ as in Theorem~\ref{th:pade_conv}, and $\abs{\cdot}$ denoting the Lebesgue measure.
\end{theorem}

\begin{proof}
Using the definition of the characteristic function and the linearity of the expected value we find
\begin{align*}
err_t
&=\abs{\phi_X(t)-\phi_{X_P}(t)}
=\abs{ \mean{ \e^{itX} } - \mean{ \e^{itX_P} } }\\
&=\abs{ \mean{ \e^{itX} - \e^{itX_P} } }
=\abs{ \int_K \left( \e^{it\mcL(\mcS(x))} - \e^{it\mcL(\pade{\mcS}(x))} \right) 
	\mathscr{F}_{k^2}(x)\,dx}\\
&\leq \abs{ \int_{K_\alpha} \left( \e^{it\mcL(\mcS(x))} - \e^{it\mcL(\pade{\mcS}(x))} \right) \mathscr{F}_{k^2}(x)\,dx}\\
&\quad +\abs{ \int_{K\setminus K_\alpha} \left( \e^{it\mcL(\mcS(x))} - \e^{it\mcL(\pade{\mcS}(x))} \right)\mathscr{F}_{k^2}(x)\,dx}.
\end{align*}
We bound the two integrals separately. For the integral over $K_\alpha$, we have
\begin{align}
\nonumber
&\abs{ \int_{K_\alpha} \left( \e^{it\mcL(\mcS(x))} - \e^{it\mcL(\pade{\mcS}(x))} \right) 
	\mathscr{F}_{k^2}(x)\,dx}\\
\label{eq:bound_I1}
&\leq \int_{K_\alpha} \abs{ \e^{it\mcL(\mcS(x))} } \mathscr{F}_{k^2}(x)  dx +
	\int_{K_\alpha} \abs{ \e^{it\mcL(\pade{\mcS}(x))} } \mathscr{F}_{k^2}(x) dx
\leq 2 \abs{K_\alpha} \sup_{x\in K_\alpha} \mathscr{F}_{k^2}(x).
\end{align}
Consider now the integral over $K\setminus K_\alpha$. Since $\e^{itx}$ is Lipschitz as a function of $x$ with constant $\abs{t}$, and $\mcL$ is Lipschitz with constant $L$, we find
\begin{align*}
&\abs{ \int_{K\setminus K_\alpha} \left( \e^{it\mcL(\mcS(x))} - \e^{it\mcL(\pade{\mcS}(x))}
	 \right)\mathscr{F}_{k^2}(x)\,dx}\\
&\leq \int_{K\setminus K_\alpha} \abs{ \e^{it\mcL(\mcS(x))} - \e^{it\mcL(\pade{\mcS}(x))} }
	\mathscr{F}_{k^2}(x) dx\\
&\leq \abs{t} \int_{K\setminus K_\alpha} \abs{ \mcL(\mcS(x)) - \mcL(\pade{\mcS(x))} }
	 \mathscr{F}_{k^2}(x)  dx\\
&\leq \abs{t}\, L \int_{K\setminus K_\alpha} \normw{ \mcS(x)-\pade{\mcS(x)} }{V}{\sqrt{\Real{z_0}}}
	 \mathscr{F}_{k^2}(x) dx.
\end{align*}
From the bound~\eqref{eq:pade_approx_S} of~Theorem~\ref{th:pade_conv}, we obtain
\begin{align}
\nonumber
&\abs{ \int_{K\setminus K_\alpha} \left( \e^{it\mcL(\mcS(x))} - \e^{it\mcL(\pade{\mcS}(x))}
	 \right)\mathscr{F}_{k^2}(x)\,dx}\\
\label{eq:bound_I2}
&\quad \leq \abs{t}\, L\,C \frac{1}{\alpha^3}  \left(\frac{\rho}{R}\right)^{M+1} \abs{K}
	\sup_{x\in K\setminus K_\alpha} \mathscr{F}_{k^2}(x).
\end{align}
The conclusion follows from inequalities~\eqref{eq:bound_I1} and~\eqref{eq:bound_I2}.  
\end{proof}

\begin{corollary}
\label{cor:errt}
Under the same assumptions as in Theorem~\ref{th:errt_bound}, it holds 
$$
\lim_{M\rightarrow \infty} err_t=0 \quad \forall t\in\R.
$$ 
In particular, there exists $C>0$ such that for any $t\in\R$  
\begin{equation*}
	err_t\leq C\abs{t}^{1/4}\left(\frac{\rho}{R}\right)^{\frac{M+1}{4}}.
\end{equation*}
\end{corollary}

\begin{proof}
	We have $\abs{K_\alpha}\leq \alpha n$, with $n\leq N$ the number of poles of $\mcS$ in $K$. From Theorem~\ref{th:errt_bound} it holds
$$
err_t\leq\inf_{\alpha>0}\left(C_1\alpha+C_2(t)\frac{1}{\alpha^3}
\left(\frac{\rho}{R}\right)^{M+1}\right),
$$
with $C_1=2n\sup_{x\in K} \mathscr{F}_{k^2}(x)$ and $C_2(t)=\abs{t}LC\abs{K}\mathscr{F}_{k^2}(x)$. 
By optimizing the expression in $\alpha$ we obtain
$$
err_t\leq C_t\left(\frac{\rho}{R}\right)^{\frac{M+1}{4}}
$$
with $C_t=C_1^{3/4}C_2(t)^{1/4}(3^{1/4}+3^{-3/4})$.
\end{proof}

This corollary establishes, in particular, uniform exponential convergence of $\phi_{X_P}$ to $\phi_X$ on any compact subset of $\R$.

\begin{remark}
	Theorem~\ref{th:errt_bound} and Corollary~\ref{cor:errt} can be generalized to derive an a priori upper bound on $err_\xi=\abs{\mean{\xi(X)}-\mean{\xi(X_P)}}$, for any continuous and bounded functional $\xi:\R\rightarrow\R$. Thus, the weak convergence of $X_P$ to $X$ follows, as $M\rightarrow+\infty$.
\end{remark}

Let us consider the case of $\partial D=\Gamma_D$. Let $K=[7,14]$ be the interval of interest (which contains three eigenvalues of the Dirichlet-Laplace operator: $8, 10, 13$), and let the wavenumber be modeled as a random variable uniformly distributed on $K$, i.e., $k^2\sim\mathcal U(K)$. Given the functional $\mcL = \normw{\cdot}{V}{\sqrt{\Real{z_0}}}$, where $z_0=10+0.5i$, we consider the random variables $X=\normw{\mcS_h(k^2)}{V}{\sqrt{\Real{z_0}}}$ and $X_P=\normw{\mcS_{h,P}(k^2)}{V}{\sqrt{\Real{z_0}}}$. We define as $\mcS_h$ the $\mathbb P^3$ finite element approximation of $\mcS$; then $\mcS_{h,P}$ is the LS-Pad\'e approximant of $\mcS_h$, centered in $z_0$ and with polynomial degrees $(M,N)$.
In Figure~\ref{fig:stochastic_samples}, we display the random variables $X$ and $X_P$ evaluated at 100 sample points uniformly distributed in $K$. When the degree of the LS-Pad\'e denominator is $N=3$, all the poles are correctly identified by the LS-Pad\'e approximant, provided that $M$ is larger than $4$.
In Figure~\ref{fig:stochastic_phi} we plot the characteristic function of the random variable $X_P$, $\phi_{X_P}(t)$, where the degrees of the Pad\'e denominator and denominator are $N=3$,  and $M=2,4,6$, respectively. The expected value has been computed by the Monte Carlo method, using $10^5$ samples.
    
\begin{figure}[htb]
\centering
\includegraphics[width=0.45\textwidth]{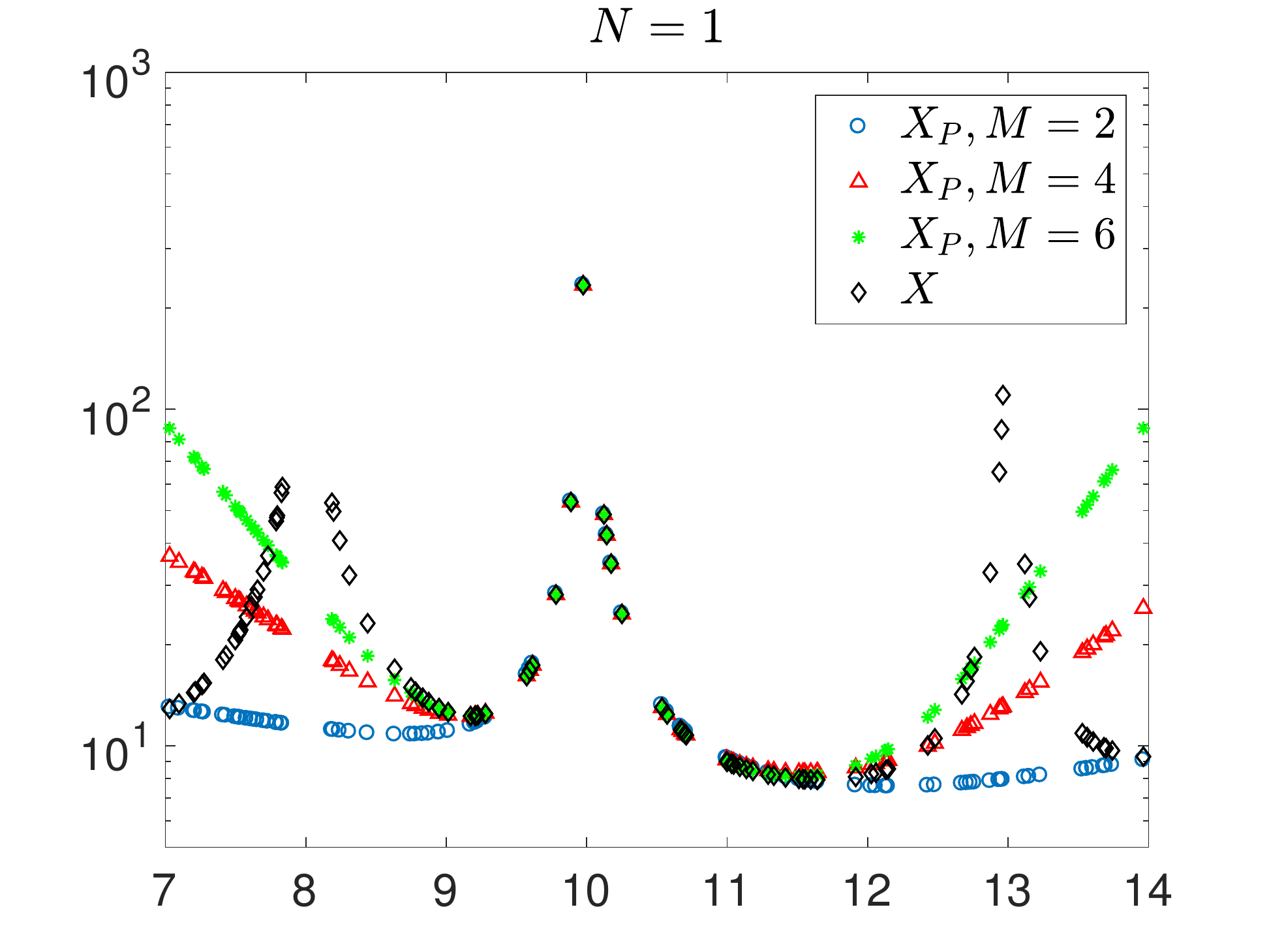}
\includegraphics[width=0.45\textwidth]{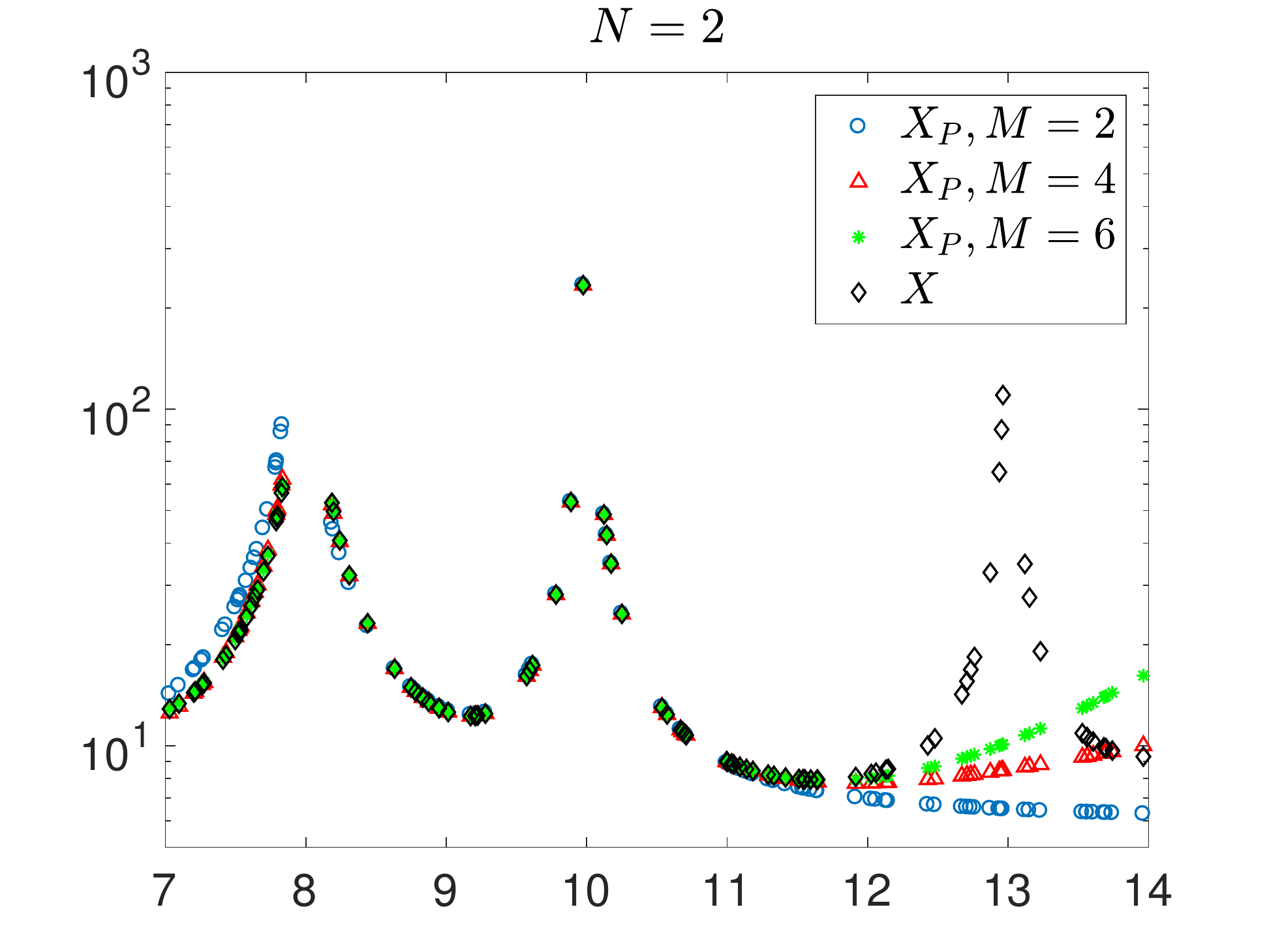}
\includegraphics[width=0.45\textwidth]{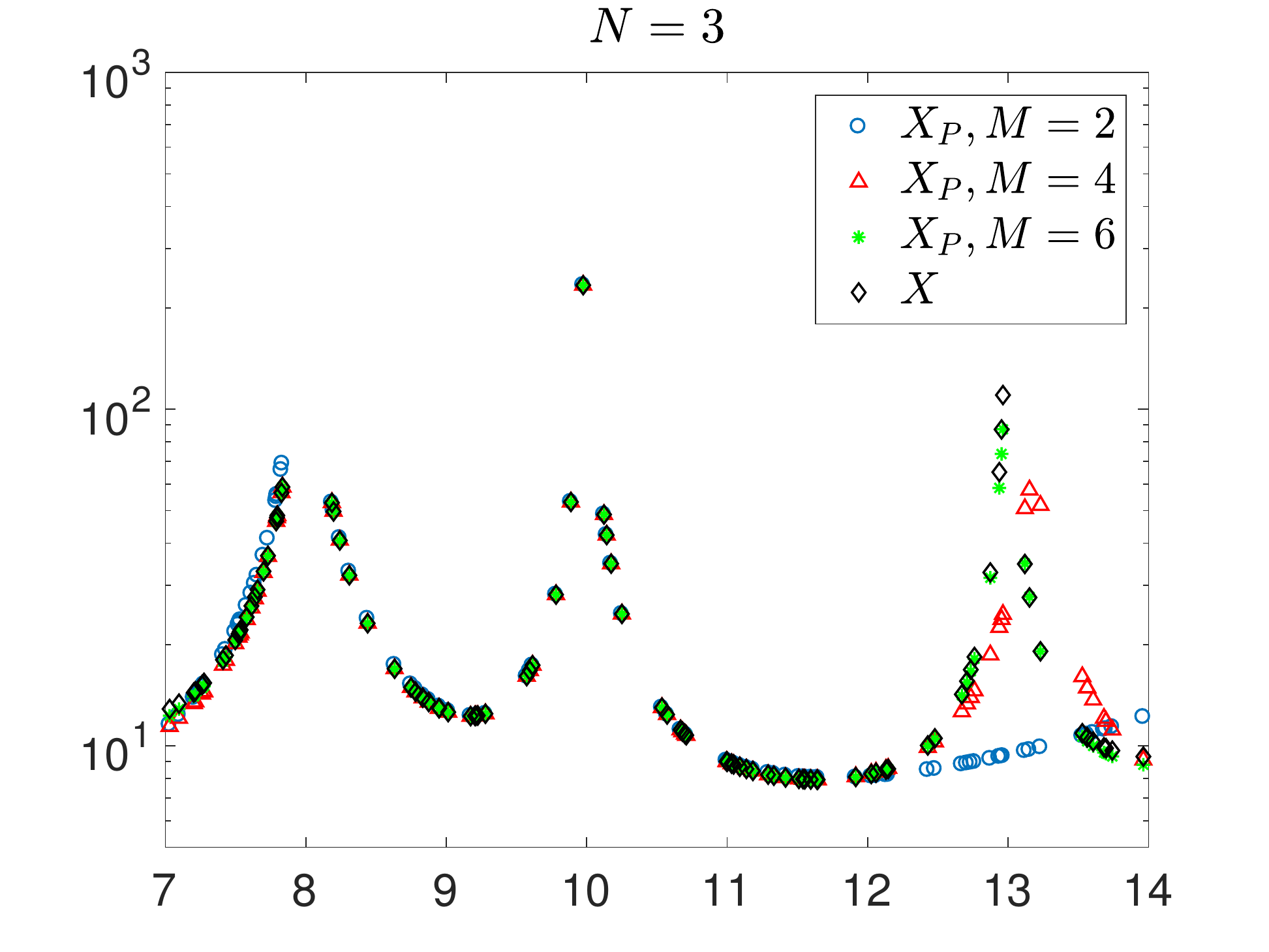}
\caption{Comparison between $X=\normw{\mcS_h(k^2)}{V}{\sqrt{\Real{z_0}}}$ and $X_P=\normw{\mcS_{h,P}(k^2)}{V}{\sqrt{\Real{z_0}}}$ evaluated at 100 sample points uniformly distributed in $K=[7,14]$).}
\label{fig:stochastic_samples}
\end{figure}

\begin{figure}[htb]
\centering
\includegraphics[width=0.7\textwidth]{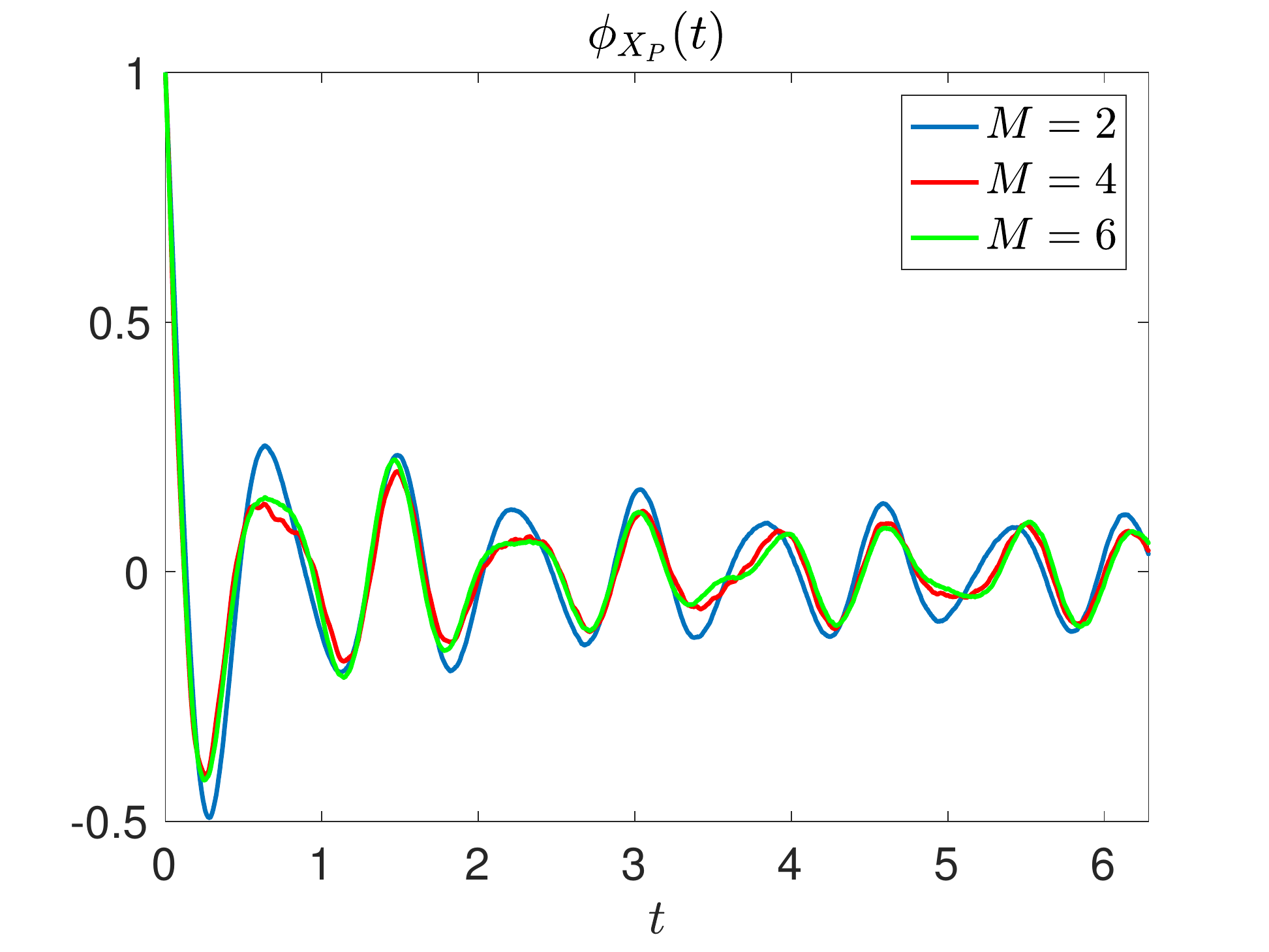}
\caption{Characteristic function $\phi_{X_P}(t)$, with $N=3$ and $M=2,4,6$.}
\label{fig:stochastic_phi}
\end{figure}



\section{Conclusions}
\label{sec:conclusions}

The present paper concerns a model order reduction method based on the single-point LS-Pad\'e approximation technique introduced in~\cite{Bonizzoni2016}. 
We have described an algorithm to compute the LS-Pad\'e approximant of the Helmholtz frequency response map, and we have explored the applicability and potentiality of 
the method via 2D numerical experiments in various contexts. 
Moreover, the time-harmonic wave equation with random wavenumber has been analyzed.

We are currently investigating the extension of the proposed methodology and of its convergence analysis to the case of \emph{multi-point} LS-Pad\'e expansions, 
where evaluations of the frequency response map $\mcS$ and of its derivatives at multiple frequencies are used. 
We believe that this technique will outperform the single-point one, when a large number of singularities of $\mcS$ need to be identified.



\end{document}